\documentclass[11pt]{article}

\usepackage[hmargin={2.3cm},vmargin=2cm,a4paper]{geometry}
\usepackage[utf8]{inputenc}
\usepackage{lmodern}
\usepackage{amsmath,amsthm,amssymb,dsfont}
\usepackage{graphicx}
\usepackage{color}
\usepackage{constants}
\usepackage[shortlabels]{enumitem}
\usepackage{comment}
\usepackage{xr-hyper}
\usepackage{hyperref}

\definecolor{Darkgreen}{rgb}{0,0.4,0}
\definecolor{gray}{gray}{0.7}

\hypersetup{colorlinks,linkcolor=Darkgreen,citecolor=Darkgreen,breaklinks=true}

\setlist[enumerate]{label=(\alph*)}

\numberwithin{equation}{section}

\newcommand{\overbar}[1]%
  {\mkern 1.5mu\overline{\mkern-1.5mu#1\mkern-1.5mu}\mkern 1.5mu}

\newcommand{\E}{\ensuremath{\mathbb{E}}}
\newcommand{\N}{\ensuremath{\mathbb{N}}}
\newcommand{\R}{\ensuremath{\mathbb{R}}}
\newcommand{\Z}{\ensuremath{\mathbb{Z}}}

\newcommand{\ei}{\ensuremath{\mathsf{ei}}}
\newcommand{\es}{\ensuremath{\mathsf{es}}}

\newcommand{\ttE}{\ensuremath{\mathtt{E}}}
\newcommand{\ttP}{\ensuremath{\mathtt{P}}}
\newcommand{\ttQ}{\ensuremath{\mathtt{Q}}}
\renewcommand{\P}{\ensuremath{\mathbb{P}}}
\renewcommand{\d}{\ensuremath{\mathrm{d}}}
\newcommand{\RM}{\ensuremath{\mathtt{RM}}}

\newcommand{\rmin}{\ensuremath{\mathcal R}}
\newcommand{\lmax}{\ensuremath{\mathcal L}}
\newcommand{\rbound}{\ensuremath{R}}
\newcommand{\lbound}{\ensuremath{L}}

\DeclareMathOperator{\LM}{\mathtt{LM}}
\DeclareMathOperator{\LC}{\mathtt{LC}}
\DeclareMathOperator{\RC}{\mathtt{RC}}
\DeclareMathOperator{\Bad}{\mathtt{B}}
\DeclareMathOperator{\Free}{\mathtt{F}}

\DeclareMathOperator*{\essinf}{ess\,inf}
\DeclareMathOperator*{\esssup}{ess\,sup}

\newcommand*\diff{\mathop{}\!\mathrm{d}} 

\newcommand{\p}{\mathbb{P}}
\newcommand{\Eprob}{\texttt{E}}

\newcommand{\given}{\, |\, }
\newcommand{\mtilde}{m}
\newcommand{\mover}{\overline{m}}

\newcommand{\F}{\mathcal{F}}

\newcommand{\tend}[2]{\displaystyle\mathop{\longrightarrow}_{#1\rightarrow#2}}

\newcommand{\parenthezises}[1]{\arabic{#1}}

\newconstantfamily{index_prob_cont_time}{
  symbol=\mathcal{T},
  format=\parenthezises,
}
\newconstantfamily{index_prob}{
  symbol=\mathcal{N},
  format=\parenthezises,
}
\newconstantfamily{constant}{
  symbol=\mathcal{C},
  format=\parenthezises,
}
\newconstantfamily{small}{
  symbol=c,
  format=\parenthezises,
}

\newtheorem{theorem}{Theorem}[section]
\newtheorem{claim}[theorem]{Claim}

\newtheorem{lemma}[theorem]{Lemma}
\newtheorem{proposition}[theorem]{Proposition}

\theoremstyle{remark}
\newtheorem{remark}[theorem]{Remark}

\theoremstyle{definition}
\newtheorem{definition}[theorem]{Definition}

\newcommand{\ignore}[1]{{}}

\begin{document} 

\author{
  Jiří Černý%
  \thanks{Department of mathematics and computer science,
    University of Basel,
    Spiegelgasse 1,
    4051 Basel, Switzerland.
    Email: \url{jiri.cerny@unibas.ch}}
  \and
  Alexander Drewitz%
  \thanks{Universität zu Köln,
    Mathematisches Institut,
    Weyertal 86--90,
    50931 Köln, Germany.
    Email: \url{{drewitz,lschmit2}@math.uni-koeln.de}}
  \and
  Lars Schmitz$^{\dag}$%
}

\title{(Un-)bounded transition fronts for the parabolic Anderson model and
  the randomized F-KPP equation}

\date{\today}

\maketitle

\begin{abstract}
  We investigate the uniform boundedness of the fronts of the solutions
  to the randomized Fisher-KPP equation and to its linearization, the
  parabolic Anderson model. It has been known that for the standard
  (i.e.~deterministic) Fisher-KPP equation, as well as for the special
  case of a randomized Fisher-KPP equation with so-called ignition type
  nonlinearity, one has a uniformly bounded (in time) transition front.
  Here, we show that this property of having a uniformly bounded
  transition front fails to hold for the general randomized Fisher-KPP
  equation. Nevertheless, we establish that this property does hold true
  for the parabolic Anderson model.
\end{abstract}


\section{Introduction}
\label{sec:intro}
We consider the random partial differential equation
\begin{align*}
  \label{eq:KPP}
  \tag{F-KPP}
  \begin{split}
    w_t(t,x) &= \frac{1}{2}w_{xx}(t,x) +  \xi(x,\omega)\, F(w(t,x)),
    \qquad (t,x)\in(0,\infty)\times\R,\ \omega\in\Omega, \\
    w(0,\cdot) &= \mathds{1}_{(-\infty,0]}.
  \end{split}
\end{align*}
In our specific setting, $(\xi(x))_{x\in\R}=(\xi(x,\omega))_{x\in\R}$,
$\omega\in\Omega$, is a stochastic process on a probability space
$(\Omega,\F,\P)$ fulfilling  suitable mixing and sample path regularity
conditions (see Section~\ref{sec:model}), and the non-linearity $F$ is
generated by the probability generating function belonging to branching
Brownian motion, see condition \eqref{eq:standard_cond} below
\eqref{eq:mom_offspr}.

The investigation of \eqref{eq:KPP} for the homogeneous case $\xi\equiv1$
has a long history, dating back to the seminal works of Fisher
\cite{fisher1937} and Kolmogorov, Petrovskii and Piscunov
\cite{kolmogorov1937}. The equation has found a plethora of applications,
such as describing the dynamics of a randomly mating diploid population
in a one-dimensional habitat, or also to model flame propagation, see
\cite{aronson_diffusion}.

It is well-known, see \cite[Theorem~14]{kolmogorov1937}, that in the
homogeneous case $\xi\equiv1$ the solution $w$ of \eqref{eq:KPP}
converges to a \emph{traveling wave solution}. More precisely, there
exists a function ${(0,\infty) \ni t \mapsto m(t)}$ such that
\begin{align}
  w(t,m(t) + \cdot\,) \tend{t}{\infty} g\quad\text{uniformly,}
  \label{eq:travelling_wave}
\end{align}
for some function $g:\R\to[0,1]$ with $g(x)\tend{x}{-\infty}0$ and
$g(x)\tend{x}{\infty}1$, and which is unique up to spatial translations.
In this context, the function $m(t)$ is usually referred to as the
\emph{position of the wave}. The convergence in
\eqref{eq:travelling_wave} implies that the front of the solution to
\eqref{eq:KPP} for the case $\xi\equiv 1$ is bounded, i.e.~for every
$\varepsilon\in(0,1/2)$ there exist $\underline{x},\overline{x}\in\R$,
such that for all $t$ large enough,
\begin{equation}
  \label{eq:bounded_front_hom}
  \inf_{x\leq \underline{x}}\, w(t,x+m(t)) \geq 1-\varepsilon
  \quad\text{and}\quad
  \sup_{x\geq\overline{x}}\, w(t,x+m(t)) \leq \varepsilon.
\end{equation}

Therefore, a question arising naturally in our context is whether a
behavior similar to \eqref{eq:bounded_front_hom} is observed in the
setting of a random nonlinearity in \eqref{eq:KPP} as well. It turns out
that in the investigation of this question, for a variety of reasons the
linearization of \eqref{eq:KPP}, which goes under the name  parabolic
Anderson model and which is of independent interest,
\begin{align*}
  \label{eq:PAM}\tag{PAM}
  \begin{split}
    u_t(t,x) &=  \frac12 u_{xx}(t,x) + \xi(x,\omega)\,  u(t,x),
    \qquad (t,x)\in(0,\infty)\times\R,\ \omega\in\Omega,\\
    u(0,x) &= u_0(x),\qquad x \in \R,
  \end{split}
\end{align*}
plays an important role as well.

\section{Model and results}
\label{sec:model}

We will assume $\xi=(\xi(x))_{x\in\mathbb{R}}$ to be a stochastic process
on a probability space $(\Omega,\F,\p)$ having   H\"{o}lder continuous
paths. I.e., there exists $\alpha=\alpha(\xi)>0$ and $C=C(\xi)>0$, such that
\begin{equation*}
  \label{eq:hoelder_cont}
  \left| \xi(x)-\xi(y) \right| \leq C\,  |x-y|^\alpha\quad \forall x,y\in \R.
  \tag{H\"OL}
\end{equation*}
We will consider throughout the standard model of $\Omega$ being the
space of H\"older continuous functions and $\mathcal F$ to be the $\sigma$-algebra
generated by point evaluations.
Furthermore, we assume the following
conditions to be fulfilled:
\begin{itemize}
  \item $\xi$ is \emph{uniformly bounded away from $0$ and $\infty$}:
  \begin{equation*}
    \label{eq:POT}
    0<
    \ei := \essinf_\omega  \xi (x,\omega ) < \esssup_\omega  \xi (x,\omega ) =: \es < \infty
    \quad \text{ for all }x\in\mathbb{R};
    \tag{BDD}
  \end{equation*}

  \item $\xi$ is \emph{stationary}: For every $h\in\mathbb{R},$
  \begin{equation*}
    \label{eq:STAT}
    (\xi(x))_{x\in\R} \text{ and } (\xi(x+h))_{x\in\R}
    \text{ have the same distribution;}
    \tag{STAT}
  \end{equation*}
  \item $\xi$ fulfills a  \emph{$\psi$-mixing} condition:  Let
  $\F_x:=\sigma(\xi(z):z\leq x)$ and $\F^y:=\sigma(\xi(z):z\geq y)$,
  $x,y\in\mathbb{R},$ and assume that there exists a continuous,
  non-increasing function $\psi:[0,\infty)\to[0,\infty)$, such that for
  all $j\leq k$ as well as integrable $\F_j$-measurable $X$ and
  integrable $\F^k$-measurable $Y,$ we have
  \begin{equation*}
    \label{eq:MIX}
    \tag{MIX}
    \begin{split}
      \big| \E\big[X-\E[X] \given \F^{k}\big]\big| &\leq \E[|X|]\cdot \psi(k-j),\\
      \big| \E\big[Y-\E[Y] \given  \F_j\big]\big| &\leq \E[|Y|]\cdot \psi({k-j}),
      \quad \text{and}   \\
      \sum_{k=1}^\infty &\psi(k)<\infty.
    \end{split}
  \end{equation*}
\end{itemize}
Note that \eqref{eq:MIX} implies the ergodicity of $\xi$ with respect to
the shift operator $\theta_y$  acting on $\Omega$ via
$\xi(\cdot)\circ\theta_y=\xi(\cdot+y)$, $y\in\R$.

In order to specify the initial conditions for \eqref{eq:PAM} under
consideration, for $\delta'\in(0,1)$ and $C'>1$ consider the condition
\begin{equation*}
  \label{eq:PAM_initial}
  \tag{PAM-INI}
  \begin{split}
    & \delta'\mathds{1}_{[-\delta',0]} \leq  u_0 \leq
    C'\mathds{1}_{(-\infty,0]},
  \end{split}
\end{equation*}
and we define the class of initial conditions to \eqref{eq:PAM} as
\begin{align*}
  \mathcal{I}_{\mathrm{PAM}}
  := \mathcal{I}_{\mathrm{PAM}}(\delta', C')
  &:= \big\{ u_0:\R\to[0,\infty)\text{ measurable}:\ u_0
    \text{ fulfills \eqref{eq:PAM_initial} for  }\delta' \text{ and }C' \big\}.
\end{align*}

In order to describe the admissible non-linearities for \eqref{eq:KPP},
let $(p_k)_{k\in\N}$ be an arbitrary sequence of reals in $[0,1]$ such that
\begin{equation}
  \sum_{k=1}^\infty p_k=1,\quad \sum_{k=1}^\infty kp_k = 2,
  \quad \text{and} \quad
  \sum_{k=1}^\infty k^2 p_k =: m_2 <\infty.
  \label{eq:mom_offspr}
\end{equation}
Then define $F:\ [0,1] \to [0,1]$ via
\begin{equation*}
  F(u) :=  (1-u) - \sum_{k=1}^\infty p_k(1-u)^k ,
  \quad u\in[0,1].
  \label{eq:standard_cond}\tag{PROB}
\end{equation*}
In passing, we note that $F'(0)=1$. The reason for considering this type
of non-linearity is its suitability for being investigated using
techniques from branching processes. In particular, the solutions to
\eqref{eq:KPP} can then be expressed as a functionals of a branching
Brownian motion, see Proposition~\ref{prop:KPPBRM}.

On top of the above, we need a further technical condition to be
fulfilled. In order to be able to formulate it, note that
Lemma~\ref{le:annealed_lgmf} states the existence of a critical velocity
$v_c \ge 0$ and Proposition~\ref{prop:lyapunov_tight} that of another
velocity $v_0 > 0;$ here, the former pertains to the characteristics of
the Lyapunov exponent while, under suitable assumptions, the latter
essentially is the speed of the front of the solutions to \eqref{eq:PAM}
and \eqref{eq:KPP}. In order for our approach to be effective, we need to
perform a change of measure that requires
\begin{equation*}
  \label{eq:VEL}
  \tag{VEL}
  v_0>v_c
\end{equation*}
to be fulfilled.
For the time being, we content ourselves with referring to Section \ref{sec:discussion}, where we argue that there do exist potentials fulfilling \eqref{eq:VEL}, alongside all other conditions required for our results to hold. For further details and a more profound discussion of condition \eqref{eq:VEL}, as well as for examples of potentials which do or do not entail \eqref{eq:VEL} to be satisfied, we refer to \cite{DreSchmi19}.

In order to investigate the position of the front, we  introduce for
$\varepsilon \in (0,1)$, $M>0$ and $t\ge 0$ the quantities
\begin{equation}
  \begin{split}
    \mtilde^\varepsilon(t) &
    := \sup\{ x \in \R \, : \, w(t,x) \ge \varepsilon\},
    \label{eq:def_front_KPP}\\
    \mtilde^{\varepsilon,-}(t) &
    :=  \inf\{ x\geq 0 \, : \,  w(t,x)\leq \varepsilon\}, \\
    \mover^M(t) &
    := \sup\{ x \in \R \, : \, u(t,x) \ge M\},\\
    \mover^{M,-}(t) &
    := \inf\{ x\geq 0 \, : \,  u(t,x)\leq M\}.
  \end{split}
\end{equation}
Note that all these quantities are random variables (and their
  distributions depend on the initial conditions of the respective
  equations).

\begin{definition}
  \label{def:transitionFronts}
  The solution to \eqref{eq:KPP} is said to have a \emph{uniformly
    bounded transition front} if for each $\varepsilon  \in (0,\frac12)$
  there exists a constant $C_\varepsilon \in (0,\infty)$ such that $\P$-a.s.,
  for all $t$ large enough we have
  \begin{equation*}
    \mtilde^\varepsilon(t) - \mtilde^{1-\varepsilon,-}(t) \le
    C_\varepsilon.
  \end{equation*}
  The solution to  \eqref{eq:PAM} is said to have a \emph{uniformly
    bounded transition front} if for all $\varepsilon, M \in (0,\infty)$
  with $\varepsilon < M$, there exists a constant
  $C_{\varepsilon, M} \in (0,\infty)$ such that $\P$-a.s.,  for all $t$
  large enough,
  \begin{equation}
  \label{eq:const_tf_PAM}
    \overline m^\varepsilon(t) - \mover^{M,-}(t) \leq C_{\varepsilon,M}.
  \end{equation}
\end{definition}
We can now state our two main results. The first one is for the solution
to \eqref{eq:PAM} and states that its transition front stays bounded
uniformly in time.

\begin{theorem}
  \label{thm:PAMBdd}
  If \eqref{eq:hoelder_cont},
  \eqref{eq:POT}, \eqref{eq:STAT}, \eqref{eq:MIX} and \eqref{eq:VEL} are fulfilled, the  solution to \eqref{eq:PAM} has a
  uniformly bounded transition front. Furthermore, for $\delta', C' >0$
  fixed,  the corresponding constant $C_{\varepsilon,M}$ in
  \eqref{eq:const_tf_PAM} is independent of
  $u_0\in\mathcal{I}_{\mathrm{PAM}}(\delta', C')$.
\end{theorem}

Our second, and more important,  main result states that an analogous
statement is in general {\em not} true for the solution to  \eqref{eq:KPP}.

\begin{theorem}
  \label{thm:FKPPunBdd}
  There exist potentials $\xi$ fulfilling \eqref{eq:hoelder_cont},
  \eqref{eq:POT}, \eqref{eq:STAT} and \eqref{eq:MIX} such that the
  transition front of the solution to \eqref{eq:KPP} is not uniformly
  bounded in time.   More precisely, such $\xi$ can be chosen so that for
  any $\delta\in(0,1)$ and any $\varepsilon>0$ we find a sequence
  $(x_n, t_n)_{n \in \N}$ in $\R \times [0,\infty)$ as well as a function
  $\varphi \in \Theta(\ln n)$ such that
  \begin{enumerate}
    \item
    $x_n, t_n \to \infty$ as $n \to \infty$, and $(x_n)_{n\in\N} \in \Theta(n)$,

    \item for all $n \in \N$,
    \begin{equation} \label{eq:unbddFront}
      \delta = w(t_n, x_n) \le w(t_n, x_n + \varphi(n)) +
      \varepsilon.
    \end{equation}
  \end{enumerate}
\end{theorem}

This means that, at least along a subsequence of times, the interval of
transition in which the solution changes from being locally unstable
($w\approx 0$) to locally stable ($w\approx 1$), grows at least
logarithmically in time as $t \to \infty.$

While the previous result will be derived using probabilistic techniques,
we will enhance it employing analytic techniques to show that the
statement of Theorem~\ref{thm:FKPPunBdd} is true even for some  ``negative
$\varepsilon$''.
In particular, this entails the non-monotonicity of the solution in space.

\begin{theorem} \label{thm:non-monotone}
  There exist potentials $\xi$ fulfilling \eqref{eq:hoelder_cont},
  \eqref{eq:POT}, \eqref{eq:STAT} and \eqref{eq:MIX}, some $\varepsilon>0$
  small enough, and sequences $(t_n')_{n\in\N} $,
  $(l_n')_{n\in\N}$ and $(r_n')_{n\in\N}$ in $[0,\infty)$ such that
  $t_n',r_n',l_n'\in \Theta(n)$, $l_n'<r_n'$ for all $n$,  $r_n-l_n\in\Theta(\ln n)$
  and for all $n\in\N$,
  \begin{equation*}
    w(t_n',l_n') \leq w(t_n',r_n')-\varepsilon.
  \end{equation*}
\end{theorem}

Let us already mention here that at a first glance, it may seem
slightly difficult to reconcile the statement of Theorem \ref{thm:PAMBdd}
with the the statements of Theorems \ref{thm:FKPPunBdd} and
\ref{thm:non-monotone}. In particular, it might seem surprising given
that oftentimes the linearization of a non-linear PDE is considered to be
a good approximation for the original PDE, at least in the domain
    where the solutions remain small. We will address this issue in
more detail towards the end of Section \ref{sec:discussion}.

\begin{remark}
  \label{rem:concl_thm}
  It will become apparent from the respective proofs that
  Theorem~\ref{thm:PAMBdd}--\ref{thm:non-monotone} have immediate
  discrete space analogues for the respective stochastic partial
  difference equations. These are obtained as follows:
  \begin{enumerate}
    \item In equations \eqref{eq:PAM} and \eqref{eq:KPP}, $x\in \R$ is
    replaced by  $x\in \Z$, and the Laplace operator $\Delta$ is replaced
    by the discrete Laplace operator
    $\Delta_{\mathrm d}  f(x)=  \frac12 (f(x+1)+f(x-1) -2f(x))$.

    \item
    The potential $(\xi(x))_{x \in \R}$ is replaced by
    $(\xi(x))_{x \in \Z}$, the assumption \eqref{eq:POT} is replaced by
    $\ei\leq \xi(x)\leq \es$ for all $x\in\Z,$ and condition
    \eqref{eq:STAT} is replaced by
    $(\xi(x))_{x\in\Z}\overset{d}{=}(\xi(x+1))_{x\in\Z}$.

    \item
    In \eqref{eq:def_front_KPP} and
    Definition~\ref{def:transitionFronts}, $x \in \R$ is again
    substituted by $x \in \Z$.
  \end{enumerate}

  Then the statements of Theorems~\ref{thm:PAMBdd},
  \ref{thm:FKPPunBdd} and \ref{thm:non-monotone} still hold verbatim.
\end{remark}

\subsection{Discussion and previous results}
\label{sec:discussion}

As already explained in the Introduction, the homogeneous case of
constant $\xi$ has been well-understood by now (and, in fact, to a much
  finer extent than illustrated in the Introduction, see e.g.\
  \cite{Bo-16} and references therein for further details). Also the
heterogeneous case of random non-linearities we are dealing with has been
investigated before. Specifically, under fairly general assumptions, the
existence and characterization of the propagation speed (i.e., the linear
  order of the position of the front
  $\lim_{t \to \infty} \mtilde^\varepsilon(t)/t$) have been derived by
Freidlin and G\"artner, see e.g.~\cite{gaertner_freidlin79} as well as
\cite[Chapter VII]{Fr-85} using large deviation principles. Incidentally,
the Feynman-Kac formula (see also Section \ref{sec:PDE} below), which
characterizes the solution to the linearization \eqref{eq:PAM}, also
played an important role in the derivation.

In the setting described in the Introduction, second order corrections to
the position $\mtilde^\varepsilon(t)$ of the front are obtained in
\cite{DreSchmi19}, where it has been shown that the suitably centered and
rescaled front fulfills an invariance principle. Again, the proof takes
advantage of analyzing \eqref{eq:PAM} first. Let us note here that in
\cite{No-11a}, a corresponding invariance principle has been derived for
non-linearities that are either ignition type or bistable; note however,
that -- as will be explained below -- on a logarithmic in time scale
these fronts behave quite differently from the fronts to \eqref{eq:KPP}
in our context. For a different and due technical reasons restricted set
of initial conditions, Nolen \cite{No-11} has derived a central limit
theorem for the position of the front of the solution to \eqref{eq:KPP}
by analytic means. The initial condition $w_0(x,\xi)$ of \cite{No-11} is
required to depend on the randomness of the environment.

When it comes to the boundedness of transition fronts, Nolen and Ryzhik
\cite{NoRy-09} consider the setting of a stationary, ergodic and bounded
$\xi.$ The nonlinearity $F$ is assumed to be of ignition-type. I.e.,
there exists $\theta\in(0,1)$ such that
\begin{equation}
  \label{eq:ignition_type}
  F(w)=0\text{ for all }w\notin(\theta_0,1),\ F(w)>0
  \text{ for all }w\in(\theta_0,1),\text{ and }F'(1)<0.
\end{equation}
They find that the solution to \eqref{eq:KPP} has a uniformly bounded
transition front, see \cite[Proposition~2.3]{NoRy-09}. Our main result
Theorem \ref{thm:FKPPunBdd} entails that condition
\eqref{eq:ignition_type} is crucial here, since otherwise one cannot
expect uniformly bounded transition fronts.

Also note that in
\cite[Theorem \ref{LD-th: log-distance breakpoint median}]{DreSchmi19}
it has been shown that in the setting of the
current article,  the front of the solution to
\eqref{eq:KPP} lags behind the front of    the solution to \eqref{eq:PAM}
at most logarithmically in $t$. More precisely, for
$\varepsilon \in (0,1),$
\begin{equation*}
  \mover^\varepsilon(t) - \mtilde^\varepsilon(t) \in O(\ln t), \quad t \to \infty.
\end{equation*}
Therefore, it immediately arises the question whether this upper
bound is sharp. Theorem \ref{thm:FKPPunBdd} provides the following
partial affirmative answer: There exists an increasing sequence $(t_n)$
of times with $t_n \in (0,\infty)$ such that
$\lim_{n \to \infty} t_n = \infty$ and a sequence $(x_n)$ of reals such
that $\overline  m^\frac12(t_n) - x_n \ge c_0 \log t_n$ such that for all
$n \in \N:$
\begin{equation*}
  w(t_n, x_n) < \frac12 \quad \text{ and (by definition) } \quad
  u(t_n, \overline  m^\frac12(t_n)) = \frac 12.
\end{equation*}

As in the homogeneous context, there are profound and interesting links
to branching processes (in random environment). In \cite{CeDr-15}, in the
setting of discrete space, invariance principles have been derived for
the position of the front of the PAM as well as the
position of the maximum of BRWRE. Furthermore, it has been shown that the
distance between these two quantities is in $O(\ln t)$ as $t \to \infty.$
In this context, a subtle but important difference to the homogeneous
setting is that the solution to \eqref{eq:KPP} and the maximum of
branching Brownian motion in random environment (BBMRE; see Section
  \ref{sec:BBM} below for the precise definition) do exhibit a slightly
more involved interrelation. In particular, neither can we directly
transfer the sub-sequential tightness result of Kriechbaum \cite{Kr-20}
for the law of the maximum of branching random walk in random environment
(BRWRE) in the context of \cite{CeDr-15} to the setting of
\eqref{eq:KPP}, nor can we directly obtain a respective non-tightness
result for BBMRE from our unbounded transition fronts for the solution to
\eqref{eq:KPP}.
Furthermore,  it is trivial that the distribution function
$w^{\xi\equiv\text{const}}(t,\cdot)$  of the maximum of a BBMRE at time
$t,$ which is the solution to \eqref{eq:KPP} with $\xi\equiv\text{const}$,
is non-increasing in space. This again is in stark contrast to
Theorem~\ref{thm:non-monotone}, which states that this is not the case
for the solution to \eqref{eq:KPP} anymore if $\xi$   exhibits ‘‘enough''
irregularity.

As already alluded to above, Theorem~\ref{thm:PAMBdd}
as well as Theorems \ref{thm:FKPPunBdd} and \ref{thm:non-monotone}  might seem slightly surprising in the light of
each other, since they imply that the front of \eqref{eq:KPP} behaves
qualitatively quite differently from that of \eqref{eq:PAM}. In this
context, note that Theorem~\ref{thm:PAMBdd} requires condition
\eqref{eq:VEL} to be fulfilled, while the potential $\xi$ satisfying the
properties stated in Theorems~\ref{thm:FKPPunBdd} and \ref{thm:non-monotone} is constructed in
\eqref{eq:def_xi} from the sole assumption $\es/\ei >2$ of
\eqref{eq:relation_es_ei}. In Section \ref{sec:nonTrivial} below, cf.\ Proposition \ref{prop:nonTriv},
we show that these conditions can be fulfilled simultaneously and hence this regime of qualitatively different behaviors for the solutions of
\eqref{eq:KPP} and \eqref{eq:PAM} is non-trivial.

 While from a
PDE point of view we lack the experience as well as a good enough control
of the fronts that would enable us to explain this phenomenon, it becomes
more tractable from a probabilistic point of view. Indeed, we will see
below, cf.~Proposition~\ref{prop:FK}, that the solution to \eqref{eq:PAM} can
be represented in terms of expectations of a Brownian motion in random
potential, i.e.\ as
\[
    u(t,x) = E_x \Big[\exp \Big\{ \int_0^t  \xi(B_s) \, {\rm d}s \Big\}
      \, \mathds{1}_{(-\infty, 0]}(B_t) \Big].
\]
Here, $x$ which are of
linear order in time $t$,  such as $\overline m^\varepsilon(t)$, turn
out to be probabilistically ‘‘costly'' in the sense that for large
$C > 0$, Brownian motion in the expectation corresponding to $u(t,x-C),$
i.e.\ starting in $x-C$ and being to the left of the origin at time $t,$
has to make less of an effort in terms of large deviations than Brownian
motion starting in $x$ and being to the left of the origin at time $t.$
Nevertheless, the former can still collect at least as high potential
values as the latter, since, typically between $x-C$ and $0$ there are
  enough locations where $\xi $ is large.
As a consequence, $u(t,x) \ll u(t,x-C)$ for $C$
large,
which at least on a heuristic level explains how the uniform boundedness
of the transition fronts to \eqref{eq:PAM} stated in Theorem~\ref{thm:PAMBdd} comes about.

On the other hand, regarding the solution to \eqref{eq:KPP} one
has a representation in terms of a maximum of branching Brownian motion
in random environment (to be introduced in Section~\ref{sec:BBM}), see
 Proposition~\ref{prop:KPPBRM}. The coupling we
will construct below in Section~\ref{ss:coupling} demonstrates that when
it comes to the displacement of this maximum from the starting site of
the process, a crucial role is played by the values of the potential in
an environment of the starting point. Exploiting this fact in a subtle
manner, we arrive at the diverging sequence of times given in Theorem
\ref{thm:non-monotone} at which the front of \eqref{eq:KPP} is getting
wider and wider. What is more, this result can be strengthened to even
deduce the non-monotonicity stated in Theorem \ref{thm:non-monotone}.

\bigskip

\noindent \emph{Open Questions:}
\begin{enumerate}
\item[(i)] We expect that the front of the
solution to \eqref{eq:KPP} shifts from exhibiting unbounded transition
fronts (essentially when $\es - \ei$ large, and maybe further conditions, cf.\ Theorem~\ref{thm:FKPPunBdd})
to exhibiting bounded transition fronts (essentially if $\es - \ei$
  small, and maybe further conditions, cf.\ \eqref{eq:travelling_wave}). While it is not clear if
‘‘small'' means ‘‘vanishes'' in this context, let us point out here
that---while periodic media are oftentimes taken to be a simple instance
for heterogeneous or random media, cf.\ also \cite{Fr-85,
  HaNoRoRy-12,lubetzky2020maximum}---it is clear from our proofs that the
phenomenon of long stretches of areas of high and low potential, which is
crucial in our proof, is not observed for periodic media.

\item[(ii)] Is there a logarithmic upper bound corresponding to the result of Theorem \ref{thm:FKPPunBdd} as well, in the sense that     $\mtilde^\varepsilon(t) - \mtilde^{1-\varepsilon,-}(t) \le C \log t$ for all $t$ large enough?

\end{enumerate}

\bigskip

\noindent \emph{Organization of the article:} In Section~\ref{sec:tools}, we recall
the well-known Feynman-Kac formula for the solutions to
\eqref{eq:KPP} and \eqref{eq:PAM}, and introduce branching Brownian
motion in random environment, which plays the role of a key tool in this
article. Section~\ref{sec:PAM} contains the proof of
Theorem~\ref{thm:PAMBdd}, together with some preparatory results concerning
the perturbation of the solution to \eqref{eq:PAM} in space and
concentration results for the logarithmic moment generating functions.
Finally, Section~\ref{sec:proof_unbounded}
deals with the proofs of the main results about the F-KPP equation,
Theorems~\ref{thm:FKPPunBdd} and~\ref{thm:non-monotone}.

This article is closely related to \cite{DreSchmi19}. While it takes
advantage of some results derived in \cite{DreSchmi19}, it also provides
suitable results such as Lemma \ref{le:space_perturb} in a natural
context, and which are also taken advantage of in \cite{DreSchmi19}.

\section{Preliminaries}
\label{sec:tools}

In this section we recall two important well-known results which are used to
prove our main theorems, and introduce the related notation.

\subsection{Feynman-Kac representation}
\label{sec:PDE}

An important tool for the investigation of the solutions to \eqref{eq:KPP} and
\eqref{eq:PAM} are their Feynman-Kac representations. Here and in what
follows, for $x\in\R$ arbitrary, we denote by $E_x$ the expectation
operator with respect to the probability measure $P_x$ under which the
process $(B_t)_{t\geq0}$ is a standard Brownian motion starting in $x$.

\begin{proposition}[Feynman-Kac formula, {\cite[(1.32)]{bramson1983convergence}}]
  Under the  assumptions of Section~\ref{sec:model}, the (unique)
  \label{prop:FK}%
  non-negative solution $u$ to \eqref{eq:PAM} is given by
  \begin{equation}
    \label{eq:feynman_kac_PAM}
    u(t,x) = E_x \Big[\exp \Big\{ \int_0^t  \xi(B_s) \, {\rm d}s \Big\}
      \, u_0(B_t) \Big],
  \end{equation}
  while the (unique) non-negative solution $w$ to \eqref{eq:KPP} fulfills
  \begin{equation}
    \label{eq:feynman_kac_KPP}
    w(t,x) = E_x \Big[\exp \Big\{ \int_0^t  \xi(B_s)F(w(t-s,B_s))/w(t-s,B_s)\diff s \Big\}
      \, w_0(B_t) \Big].
  \end{equation}
\end{proposition}
\begin{remark}
  In fact, we will take \eqref{eq:feynman_kac_PAM} and
  \eqref{eq:feynman_kac_KPP}   as the definition of the solution to
  \eqref{eq:PAM} and \eqref{eq:KPP}, respectively. Indeed, while the
  function \eqref{eq:feynman_kac_PAM} is given explicitly, there exists a
  unique function satisfying \eqref{eq:feynman_kac_KPP} (see e.g.\
    \cite[Theorem~7.4.1]{Fr-85}). If the solution to \eqref{eq:PAM} and
  \eqref{eq:KPP} exist, it can be shown (see e.g.\
    \cite[Corollary~4.4.5]{karatzas_shreve} for \eqref{eq:PAM} and
    \cite[(1.4), p.\ 354, and (a), p.\ 355]{Fr-85} for \eqref{eq:KPP})
  that they satisfy \eqref{eq:feynman_kac_PAM} and
  \eqref{eq:feynman_kac_KPP}, respectively.
\end{remark}

\subsection{Branching Brownian motion in random environment }
\label{sec:BBM}
A key tool for proving Theorems \ref{thm:FKPPunBdd}
and~\ref{thm:non-monotone} is the correspondence between the solution to
\eqref{eq:KPP} and branching Brownian motion in random environment, cf.\
Proposition~\ref{prop:KPPBRM} below. Branching Brownian motion in random
environment $\xi$ (BBMRE) started at $x\in \mathbb R $ is defined as
follows:  Conditionally on the realization of $\xi$,  we place one
particle at $x$ at time $0$. As time evolves, all particles move
independently according to standard Brownian motion. In addition, and
independently of everything else, while at $y$, a particle splits at rate
$\xi(y)$.
Once a particle splits, this particle is removed and, randomly and
independently from everything else with probability $p_k$, replaced by $k$
new particles that are put at the position $y$ of the removed particle.
These $k$ new particles evolve independently according to the same
diffusion-branching mechanism as the remaining particles. This defines
\emph{branching Brownian motion in the branching environment} $\xi$ with
offspring distribution~$(p_k)$. For every $x\in\R$ and  $\xi$,
$\Eprob_x^\xi$ denotes the corresponding expectation of the probability
measure $\ttP_{x}^\xi$ of a BBMRE, starting in $x$.


If the respective BBMRE is evident from the context, we
use $N(t)$ to denote the set of particles alive at time $t$ in this
BBMRE. For any particle $Y \in N(t)$, we denote by $(Y_s)_{s \in [0,t]}$
the trajectory of itself and its ancestors up to time $t$. We will also
call $(Y_s)_{s \in [0,t]}$ the \emph{genealogy} of $Y$. For $t\ge 0$ and
$x\in \R$, we define
\begin{equation}
  \label{eq:LR}
  \begin{split}
    N^\ge(t,x) := \{ Y \in N(t) \, : Y_t \ge x\}
    \quad \text{ and } \quad
    N^\le(t,x) := \{ Y \in N(t) \, : Y_t \le x\}
  \end{split}
\end{equation}
as the number of particles in the process at time $t$ which are located
to the right or to the left of $x$. Furthermore, in a slight abuse of
notation, we also use $N$ to denote an entire BBMRE process.

To complete the list of notation, for a stochastic process
$X=(X_t)_{t\geq0}$ and some Borel set $B\subset\R,$ we denote
$H_B(X):=\inf\{ t\geq0: X_t\in B \}$ and set $H_x(X):=H_{\{x\}}(X)$,
$x\in\R$. For a particle $Y\in N(t)$ of a BBMRE, we set
$H_B(Y)=\inf\{ s \in [0,t]: Y_s\in B \}$, where $(Y_s)_{s\geq0}$ is the
genealogy of $Y$ and as usual $\inf \emptyset = \infty.$

\section{Boundedness of the front for PAM}
\label{sec:PAM}

In this section we show our first main result, the boundedness of the
front for the equation \eqref{eq:PAM}, that is Theorem~\ref{thm:PAMBdd}.

\subsection{A perturbation estimate}

The main tool in the
proof is a space perturbation result for the solution to
\eqref{eq:PAM} in a regime of sub-linear perturbation, see
Lemma~\ref{le:space_perturb} bellow.

To state this lemma we need to introduce some notation.
Let $\zeta(x):=\xi(x)-\es\leq 0$ with $\es$ defined in \eqref{eq:POT}. For
$\eta<0,$ define the logarithmic moment generating function as well as
the related quantities
\begin{equation}
  \label{eq:def_L_quant}
  \begin{split}
    L_{x}^{\zeta}(\eta)&:=\ln E_{x}\Big[ \exp\Big\{ \int_0^{H_{\lceil x \rceil
            -1}}( \zeta(B_s)+\eta )\diff s \Big\} \Big],\quad x\in\R,\\
    \overline{L}_{x}^\zeta(\eta)&:= \frac{1}{x} \ln E_x\Big[ \exp\Big\{
      \int_0^{H_0}(\zeta(B_s)+\eta)\diff s \Big\} \Big],\quad x>0,\\
    L(\eta) &:= \E\big[ L_1^\zeta(\eta) \big],\\
    S_x^{\zeta,v}(\eta) &:= x\Big(\frac{\eta}{v}-\overline{L}_{x}^\zeta(\eta)
    \Big),\quad x>0,\ v>0.
  \end{split}
\end{equation}
Some elementary properties of these functions are recalled in the
Appendix. Here we note that, under the assumptions \eqref{eq:POT},
\eqref{eq:STAT}, and \eqref{eq:MIX} on the potential $\xi$, we have
$\E[ \overline{L}_x^\zeta(\eta) ]=L(\eta)$ for all $\eta<0$ and all $x>0$.
Further observe that using the strong Markov property one easily shows
that for any $x\ge 1$,
\begin{equation}
  \label{eqn:sumnotationa}
  x \overline L_x^\zeta (\eta)
  =  L_x^\zeta(\eta) + \sum_{i=1}^{\lceil x\rceil-1}L_i^\zeta(\eta)
   =: \sum_{i=1}^x L_i^\zeta(\eta),
\end{equation}
where the last equality should be seen as the definition of the
sum on the right-hand side. For convenience, for $1\le x\le y$,
we also define
\begin{equation}
  \label{eqn:sumnotationb}
  \sum_{i=x+1}^y L_i^\zeta(\eta) := \sum_{i=1}^y
  L_i^\zeta(\eta) - \sum_{i=1}^x L_i^\zeta(\eta),
  \qquad \text{and} \qquad
  \sum_{i=y +1}^x L_i^\zeta(\eta)  :=-\sum_{i=x+1}^y L_i^\zeta(\eta).
\end{equation}
Furthermore, it essentially follows from
Lemma~\ref{le:quenched_lmgf}(b) that
\begin{equation}
  \label{eqn:equic}
  \parbox{11cm}{
    $\big(\eta \mapsto L_x^\zeta(\eta):x\in\R, \zeta\in \Omega
      \text{ with } \ei -\es \le \zeta \le 0 \big)$
    is a family of equicontinuous functions on every compact interval
    $I\subset(-\infty,0)$.}
\end{equation}

We further define tilted probability measures under which the process
$(B_t)_{t\geq0}$ moves {\em on average} with speed $v$ up to time $t$,
cf.\ \eqref{eq:HitExpect} below. We start with introducing the family of
tilted probability measures
\begin{equation}
  \label{eq:def_Px}
  P_x^{\zeta,\eta}(\cdot):=\exp\big\{-x\overline{L}_x^\zeta(\eta)\}\cdot
  E_x\Big[ \exp \Big\{ \int_0^{H_0}\big( \zeta(B_s)+\eta \big)\diff s \Big \};\
    \cdot\ \Big], \quad x > 0,
\end{equation}
on the space of continuous functions mapping the (initial) argument $0$ to $x$ and   vanishing only at their (variable) terminal argument.
We denote the corresponding expectation operator by $E_x^{\zeta,\eta}$.
Then we fix a compact interval $V \subset (v_c,\infty)$ (see
  Lemma~\ref{le:annealed_lgmf} (d) for the notation) containing
$v_0$ in its interior.
It is known, see Lemma~\ref{le:concEtaEmpLT_tight},
that there exists a compact
interval $\triangle\subset(-\infty,0)$, such that
$\p$-a.s., for all $t$ large enough and all $v\in V$, there exists a unique
$\eta_{vt}^{\zeta}(v)\in\triangle$ fulfilling
\begin{align}
  \label{eq:HitExpect}
  E_{vt}^{\zeta,\eta^\zeta_{vt}(v)}[H_0] & = vt \big( \overline{L}_{vt}^{\zeta}
  \big)'(\eta_{vt}^{\zeta}) = t.
\end{align}
As consequence, there exists a $\p$-a.s.\ finite random position
$\Cl[index_prob]{constHn}=\Cr{constHn}(\xi,V,\triangle)$ such that the event
\begin{equation}
  \label{eq:H_n}
  \mathcal{H}_x:=\mathcal{H}_x(V,\triangle) := \big\{ \eta_x^\zeta(v) \in
  \triangle \text{ for all }v\in V \big\} \quad\text{ occurs for all }x\geq
  \Cr{constHn}.
\end{equation}
Further, by Lemma~\ref{le:annealed_lgmf} (d) there exists
$\overline{\eta}(v)<0$, $v\in V$, such that
\begin{equation*}
  L'(\overline{\eta}(v)) = \frac1v.
\end{equation*}
Finally,
we have that
\begin{equation} \label{eq:triangleLip}
  \overline{\eta}(V)\subset\triangle \quad \text{ and }
  \quad \overline{\eta} \text{ is uniformly Lipschitz continuous on } V,
\end{equation}
cf.\ \cite[\eqref{LD-eq:defVTriangle}
  and below \eqref{LD-eq: conv speed eta_n(v')-eta(v)}]{DreSchmi19}.

We can now state our main perturbation lemma.
\begin{lemma}
  \label{le:space_perturb}
  Let $\varepsilon(t)$ be a positive function such that $\varepsilon(t)\to0$ and
  $\frac{t\varepsilon(t)}{\ln t}\to\infty$  as $t\to\infty$. Then for all $\delta>0$
  there exists $C=C(\delta)>0$ such that $\p$-a.s., for all
  $u_0\in\mathcal{I}_{\mathrm{PAM}}$ we have
  \begin{enumerate}
    \item
    \begin{equation}
      \label{eq:space_pert_genau}
      \hspace{-10mm}\limsup_{t\to\infty}\ \sup\left\{  \left|
        \frac{1}{h}\ln\left( \frac{u(t,vt+h)}{u(t,vt)} \right) - L\big(
          \overline{\eta}(v) \big)  \right|: (v,h)\in\mathcal{E}_t\right\}
      \leq \delta,
    \end{equation}
    where
    $\mathcal{E}_t:=\left\{ (v,h):\ v,
      v+\frac{h}{t}\in V,\ C(\delta )\ln t\leq |h|\leq t\varepsilon(t)  \right\}$.

    \item
    Let $\varepsilon(t)$ be a positive function such that $\varepsilon(t)\to0$.
    Then there exists a constant $\Cl{const_space_perturbation}<\infty$
    and a $\p$-a.s.~finite random variable
    $\Cl[index_prob_cont_time]{T_index_spacepert}$ such that  for all
    $t\geq\Cr{T_index_spacepert}$, uniformly in
    $0\leq h\leq t\varepsilon(t)$, $v\in V$, $v+\frac{h}{t}\in V$ and
    $u_0\in\mathcal{I}_{\mathrm{PAM}}$ we have
    \begin{equation}
      \label{eq:space_pert_ungenau}
      \Cr{const_space_perturbation}^{-1}e^{-\Cr{const_space_perturbation}h}\,
      u(t,vt) \leq u(t,vt+h) \leq
      \Cr{const_space_perturbation}e^{-h/\Cr{const_space_perturbation}}\, u(t,vt).
    \end{equation}
  \end{enumerate}
\end{lemma}

\begin{remark}
  Lemma~\ref{le:space_perturb} is a continuous-space version of
  \cite[Lemma~5.1]{CeDr-15} and its proof follows similar lines.
  We will only need part (b) of this lemma in this paper,
  for the proof of Theorem~\ref{thm:PAMBdd}. Part (a) is required in
  \cite{DreSchmi19}, where  an invariance principle for $\overline m(t)$
  is proved. However, the proof of Lemma~\ref{le:space_perturb}(b)
  heavily builds on that of (a), which is why it is natural to provide it
  here.
\end{remark}

\begin{proof}[Proof of Lemma \ref{le:space_perturb}]
  (a) It is shown in \cite[{\eqref{LD-eq:PAM_sandwich}}]{DreSchmi19}
  that  there exists a constant $\widetilde{C}=\widetilde{C}(\delta',C')$,
  with $\delta',C'$ from \eqref{eq:PAM_initial}, such that for all
  $u_0\in\mathcal{I}_{\mathrm{PAM}}$, all $v\in V,$ and all $t$ large enough
  \begin{equation}
    \label{eq:PAM_sandwich}
    \widetilde{C}^{-1} u^{\mathds{1}_{(-\infty,0]}}(t,vt) \leq u^{{u_0}}(t,vt) \leq C' u^{\mathds{1}_{(-\infty,0]}}(t,vt).
  \end{equation}
  where $u^{u_0}$ denotes the solution to \eqref{eq:PAM} with initial
  condition $u_0$. Therefore, in order to establish
  \eqref{eq:space_pert_genau}, it is enough to consider
  $u_0=\mathds{1}_{(-\infty,0]}$.

  For this $u_0$, the solution to
  \eqref{eq:PAM} can be represented by the Feynman-Kac formula (see
    Proposition~\ref{prop:FK})
  \begin{equation*}
    u(t,vt)
    =E_{vt}\Big[ e^{\int_0^t \xi(B_s)\diff s}; B_t\leq 0 \Big].
  \end{equation*}
  If follows from
  \cite[Corollary~\ref{LD-cor: Y_v approx exp part} and \eqref{LD-eq:Y_v^approxStab}]{DreSchmi19}
  that if $\mathcal H_{vt}$ occurs, then, up to a universal multiplicative
  constant, this can be approximated by
  \begin{equation*}
    E_{vt}\Big[ e^{\int_0^{H_0} \xi(B_s)\diff s}; H_0\le t \Big].
  \end{equation*}
  We now consider $t$ large enough such that $\mathcal{H}_{vt}$ occurs for all
  $v\in V$. Taking $(v,h)\in\mathcal{E}_t$ and defining
  $v':=v+\frac ht\in V$, we see that $\mathcal H_{v't}$ occurs as well.
  Therefore
  the fraction in \eqref{eq:space_pert_genau}, up to a positive
  multiplicative constant, is equal to
  \begin{align*}
    \frac{E_{v't}\big[ \exp\big\{\int_0^{H_0}\zeta(B_s)\diff s\big\}; H_0\leq t
        \big]}{E_{vt}\big[ \exp\big\{\int_0^{H_0}\zeta(B_s)\diff s\big\}; H_0\leq t
        \big]}=\frac{E_{v't}\big[
      \exp\big\{\int_0^{H_0}\big(\zeta(B_s)+\eta^\zeta_{v't}(v')\big)\diff s\big\}
      e^{-\eta^\zeta_{v't}(v') H_0}; H_0\leq t  \big]}{E_{vt}\big[
    \exp\big\{\int_0^{H_0}\big(\zeta(B_s)+\eta^\zeta_{vt}(v)\big)\diff
    s\big\}e^{-\eta^\zeta_{vt}(v) H_0}; H_0\leq t  \big]}.
  \end{align*}
  Using that $E_{vt}^{\zeta,\eta_{vt}(v)}[H_0]= E_{v't}^{\zeta,\eta_{v't}(v')}[H_0]=t$,
  recalling \eqref{eq:def_L_quant} and \eqref{eq:def_Px}, the latter
  fraction can be written as
  \begin{align*}
    \frac{ E_{v't}^{\zeta,\eta_{v't}^{\zeta}(v')}
      \big[ e^{ -\eta_{v't}^{\zeta}(v')( H_0 -
            E_{v't}^{\zeta,\eta_{v't}^\zeta(v')}[H_0] ) };
        H_0 - E_{v't}^{\zeta,\eta_{v't}^\zeta(v')}[H_0] \leq 0 \big]  }
    { E_{vt}^{\zeta,\eta_{vt}^{\zeta}(v)} \big[ e^{ -\eta_{vt}^{\zeta}(v)
          ( H_0 - E_{vt}^{\zeta,\eta_{vt}^\zeta(v)}[H_0] ) };
        H_0 - E_{vt}^{\zeta,\eta_{vt}^\zeta(v)}[H_0] \leq 0 \big]
    }
    \cdot \exp\Big\{  S_{vt}^{\zeta,v}(\eta_{vt}^{\zeta}(v)) -
      S_{v't}^{\zeta,v'}(\eta_{v't}^{\zeta}(v'))   \Big\},
  \end{align*}
  Since $\mathcal{H}_{vt}$ and $\mathcal{H}_{v't}$ occur, the first
  fraction is bounded from below and above by positive constants (see
    \cite[{Lemma~\ref{LD-le:LCLTHitting}}]{DreSchmi19}). The logarithm of
  the second factor divided by $h$ can be written as
  \begin{equation}
    \label{eq:sum}
    \frac{1}{h} \big(S_{vt}^{\zeta,v}(\eta_{vt}^{\zeta}(v)) -
      S_{v't}^{\zeta,v'}(\eta_{vt}^{\zeta}(v)) \big) + \frac{1}{h} \big(
      S_{v't}^{\zeta,v'}(\eta_{vt}^{\zeta}(v))  -
      S_{v't}^{\zeta,v'}(\eta_{v't}^{\zeta}(v')) \big).
  \end{equation}

  We claim that the second summand in \eqref{eq:sum} tends to $0$ uniformly
  in $(v,h)\in\mathcal{E}_t$ as $t\to\infty$, $\p$-a.s.
  Indeed, by a Taylor expansion we get
  \begin{equation}
    \label{eq:taylor}
    \begin{split}
      S_{v't}^{\zeta,v'}&(\eta_{vt}^{\zeta}(v)) -
      S_{v't}^{\zeta,v'}(\eta_{v't}^{\zeta}(v'))
      \\&=
      (S_{v't}^{\zeta,v'})'(\eta_{v't}^{\zeta}(v'))\cdot \big( \eta_{vt}^\zeta(v) -
        \eta_{v't}^\zeta(v') \big) +
      \frac{1}{2}(S_{v't}^{\zeta,v'})''(\widetilde{\eta}) \big( \eta_{vt}^\zeta(v) -
        \eta_{v't}^\zeta(v') \big)^2
    \end{split}
  \end{equation}
  for some $\widetilde{\eta}\in\triangle$ between $\eta_{v't}^{\zeta}(v')$
  and $\eta_{vt}^{\zeta}(v)$. By \eqref{eq:HitExpect} and
  Lemma~\ref{le:quenched_lmgf} we have
  $(S_{v't}^{\zeta,v'})'(\eta_{v't}^{\zeta}(v'))=0$.
  Lemma~\ref{le:annealed_lgmf} (b) entails that
  $( \overline{L}^\zeta_{v't})''(\cdot)$ is uniformly bounded on
  $\triangle$ and thus
  \begin{equation}
    \label{eq:S_bound}
    (S_{v't}^{\zeta,v'})''(\widetilde{\eta})= -v't(
    \overline{L}^\zeta_{v't})''(\widetilde{\eta})\in [-v'tc_1^{-1},-v'tc_1].
  \end{equation}
  Furthermore, by  Lemma~\ref{le:PertEta_n_tight} we have
  \begin{equation}
    \label{eq:eta_nPert1}
    \big| \eta_{vt}^\zeta(v) - \eta_{v't}^\zeta(v) \big| \leq c_2
    \frac{|h|}{vt} \leq c_3\frac{|h|}{t},
  \end{equation}
  and by \cite[\eqref{LD-eq:conc_eta_vt}]{DreSchmi19}
  \begin{align}
    \label{eq:eta_nPert2}
    \big| \eta_{v't}^\zeta(v)-\eta_{v't}^\zeta(v')  \big|\leq c_4
    |v-v'|=c_4\frac{|h|}{t}.
  \end{align}
  Thus, for all $t$ large enough, uniformly in $(v,h)\in\mathcal{E}_t$, we get
  \begin{equation}
    \label{eq:spacePert2ndsummand}
    \Big|\frac{1}{h} \big( S_{v't}^{\zeta,v'}(\eta_{vt}^{\zeta}(v))  -
    S_{v't}^{\zeta,v'}(\eta_{v't}^{\zeta}(v')) \big)\Big| \leq c_5
    \frac{|h|}{t}\leq \varepsilon(t),
  \end{equation}
  which tends to zero by assumption.

  It remains to show  convergence of
  the first summand in \eqref{eq:sum}. We first note that, using the
  notation introduced in \eqref{eqn:sumnotationa}, \eqref{eqn:sumnotationb},
  \begin{equation}
    \label{eq:sum2}
    \frac1h \big( S_{vt}^{\zeta,v}(\eta_{vt}^{\zeta}(v)) -
      S_{v't}^{\zeta,v'}(\eta_{vt}^{\zeta}(v)) \big)
    = \frac1h\sum_{i= vt+1}^{ v't}
    L_i^\zeta(\eta_{vt}^{\zeta}(v))  .
  \end{equation}
  To finish the proof, we will use the following lemma. Recall
  $\Cr{constHn}$ from definition \eqref{eq:H_n} and let
  $\varepsilon^*(t):=\sup_{s\in[\lfloor t\rfloor,\lceil t\rceil]}\varepsilon(s)$.

  \begin{lemma}[cf.~{\cite[Claim~5.2]{CeDr-15}}]
    For every $\delta>0$ and every $q\in\N$, there exists
    $C_0=C_0(q,\delta)>0$ such
    that  for all $t\ge 1$
    \begin{equation}
      \label{eq:spacePertLemma}
      \p\Big(
        \sup_{\substack{C_0\ln \lfloor t\rfloor \leq |h|
            \leq \lceil t\rceil\cdot\varepsilon^*(t),\\ v\in V}}
        \Big| L(\overline{\eta}(v))
        - \frac{1}{h}\sum_{i=vt+1}^{ vt+h}L_i^\zeta(\eta_{v t}^{\zeta}(v))
        \Big| >\delta, \big(
        vt\geq \Cr{constHn}\ \forall v\in V \big) \Big) \leq c t^{-q}.
    \end{equation}
    \label{le:spacePertLemma}
  \end{lemma}

  To not disturb the flow of reading, we postpone the proof of
  Lemma~\ref{le:spacePertLemma} to Section~\ref{sec:proofSpacePert}
  below. We let $A_t$ be the first event and $B_t$ be the second event on
  the left-hand side of \eqref{eq:spacePertLemma}. By
  Lemma~\ref{le:spacePertLemma} with $q=2$ and $C_0=C_0(2,\delta/3)$,
  $\sum_n \P(A_{n},B_n)<\infty$ and thus,  by the first Borel-Cantelli
  lemma, $\p$-a.s.\ the event $A_n \cap B_n$ occurs only finitely often.
  Because $\Cr{constHn}$ is $\p$-a.s.\ finite, we get
  \begin{equation}
    \label{eq:proof_1}
    \sup_{\substack{C_0\ln t \leq |h|\leq t\varepsilon(t),\\ v\in V}} \Big|
    L(\overline{\eta}(v))
    - \frac{1}{h}\sum_{i=v\lfloor t\rfloor+1}^{ v\lfloor t\rfloor+h}
    L_i^\zeta(\eta_{v\lfloor t\rfloor}^{\zeta}(v)) \Big|
    \leq \frac{\delta}{3}
  \end{equation}
  $\p$-a.s.\ for all $t$ large enough.

  To bound the right-hand side of \eqref{eq:sum2}, we need to replace
  $v\lfloor t\rfloor$ in \eqref{eq:proof_1} by $vt$.  First note that for
  all $x,y,z\in\R$ such that  $x\leq y\leq z$, due to the strong Markov
  property at $H_y$, similarly as \eqref{eqn:sumnotationa}, we have
  $\sum_{i=x+1}^z L_i^\zeta(\eta)= \sum_{i=x+1}^yL_i^\zeta(\eta) + \sum_{i=y+1}^z L_i^\zeta(\eta)$
  and thus
  \begin{align*}
    \sum_{i=v\lfloor t\rfloor+1}^{v\lfloor t\rfloor+h}L_i^\zeta(\eta) -
    \sum_{i=vt+1}^{vt+h} L_i^\zeta(\eta)
    & = \ln E_{vt}\big[ e^{\int_0^{H_{v\lfloor t\rfloor}}(\zeta(B_s)+\eta)} \big]
    - \ln E_{vt+h}\big[ e^{\int_0^{H_{v\lfloor t\rfloor + h}}(\zeta(B_s)+\eta)} \big].
  \end{align*}
  By \cite[(2.0.1), p.~204]{handbook_brownian_motion} we have
  \begin{equation}
    \label{eq:BS}
    \ln E_x\big[ e^{-c H_y}\big] = \sqrt{2c}|y-x|, \qquad \text{for all $c\geq 0$ and
      $x,y\in\R$}.
  \end{equation}
  Therefore, for all $t$ large enough, for every
  $\eta \in \triangle \subset (-\infty,0)$ and
  $0\geq \zeta(x)\geq -( \es-\ei)$,
  \begin{equation}
    \label{eq:proof_2}
    \sup_{\substack{C_0\ln t \leq |h|\leq t\varepsilon(t), \\ v\in V}}\Big|
    \frac{1}{h} \Big(\sum_{i=v\lfloor t\rfloor+1}^{v\lfloor
        t\rfloor+h}L_i^\zeta(\eta)
      - \sum_{i=vt+1}^{vt+h}
      L_i^\zeta(\eta) \Big)\Big|
    \leq \frac{2v\sqrt{2(|\eta |+(\es-\ei))}}{C_0\ln t}
    \leq \frac{\delta}{3}.
  \end{equation}
  In particular, since $\eta_{v\lfloor t\rfloor}^\zeta(v)\in\triangle\subset(-\infty,0)$
  (cf.\ \eqref{eq:triangleLip}), \eqref{eq:proof_2} holds with $\eta $
  replaced by $\eta_{v\lfloor t\rfloor}^\zeta(v)$.
  Moreover,
  By Lemma~\ref{le:PertEta_n_tight}, there exists $C>0$ such that  $\p$-a.s.
  for all $x$ large enough we have
  $\sup_{v\in V} | \eta^\zeta_{x+h}(v)-\eta^\zeta_x(v)|\leq C\frac{h}{x}$
  for all $h\in[0,x]$. Using the equicontinuity \eqref{eqn:equic} of
  $L_x^\zeta(\cdot)$, we get that $\p$-a.s.\ for all $t$ large enough,
  \begin{align}
    \sup_{\substack{C_0\ln t \leq |h|\leq t\varepsilon(t), \\ v\in V}} \Big|
    \frac{1}{h} \sum_{i=vt+1}^{vt+h}\big( L_i^\zeta(\eta_{v\lfloor
          t\rfloor}^\zeta(v)) - L_i^\zeta(\eta_{vt}^\zeta(v)) \big) \Big| \leq
    \frac{\delta}{3}.
    \label{eq:proof_3}
  \end{align}
  Applying the triangle inequality to the inequalities
  \eqref{eq:proof_1}--\eqref{eq:proof_3}, the absolute value of the
  difference of the right-hand side of \eqref{eq:sum2} and
  $L(\overline{\eta}(v))$ is bounded from above by $\delta$,
  uniformly in $(v,h)\in\mathcal{E}_t$ for all $t$ large enough,
  completing the proof of claim (a).

  (b) Analogously to the first steps in the proof of (a), it is enough to
  consider the case $u_0=\mathds{1}_{(-\infty,0]}$, and then to show
  that the expression in \eqref{eq:sum} is bounded from above and below
  by negative constants, uniformly for all
  $0< h\leq t\varepsilon(t).$ Performing the same calculations as in the
  proof of (a), i.e.~using equations \eqref{eq:taylor} to
  \eqref{eq:eta_nPert2}, one can observe that the second summand in
  \eqref{eq:sum} is contained in the interval $[-c_5\frac{h}{t},0]$ for
  $c_5$ from \eqref{eq:spacePert2ndsummand} uniformly for all $v\in V$
  and $v':=v+\frac{h}{t}\in V$ and all $t$ large enough.

  For the first summand in \eqref{eq:sum}, we mention that due to the strong
  Markov property at time $H_{vt}$, we have
  \begin{align*}
    S_{vt}^{\zeta,v}(\eta_{vt}^{\zeta}(v)) -
    S_{v't}^{\zeta,v'}(\eta_{vt}^{\zeta}(v)) & = \ln E_{vt+h} \big[
      e^{\int_0^{H_0}(\zeta(B_s)+\eta_{vt}^{\zeta}(v))\diff s}  \big] -  \ln E_{vt}
    \big[ e^{\int_0^{H_0}(\zeta(B_s)+\eta_{vt}^{\zeta}(v))\diff s} \big]
    \\& =\ln E_{vt+h} \big[
      e^{\int_0^{H_{vt}}(\zeta(B_s)+\eta_{vt}^{\zeta}(v))\diff s}  \big].
  \end{align*}
  Using \eqref{eq:BS}, \eqref{eq:POT}  and
  $\eta_{vt}^{\zeta}(v)\in\triangle\subset(-\infty,0)$, for all $t$ large
  enough, we get
  \begin{equation}
    -\sqrt{2(|\eta_{vt}^{\zeta}(v)|+\es-\ei)}h
    \leq \ln E_{vt+h} \big[ e^{\int_0^{H_{vt}}(\zeta(B_s)+\eta_{vt}^{\zeta}(v))\diff s}  \big]
    \leq -\sqrt{2|\eta_{vt}^{\zeta}(v)|}h
  \end{equation}
   and we can conclude.
\end{proof}

\subsection{Proof of Lemma~\ref{le:spacePertLemma}}
\label{sec:proofSpacePert}

To finish the proof of Lemma~\ref{le:space_perturb}, we still have to
provide the proof of Lemma~\ref{le:spacePertLemma}.

\begin{proof}[Proof of Lemma~\ref{le:spacePertLemma}]
  We decompose the difference in \eqref{eq:spacePertLemma} as
  \begin{equation}
    \label{eq:dec}
    L(\overline{\eta}(v)) - \sum_{i= vt+1}^{vt+h}
    L_i^\zeta(\eta_{vt}^{\zeta}(v))
      = L(\overline{\eta}(v)) - \sum_{i= vt+1}^{vt+h}
    L_i^\zeta(\overline{\eta}(v)) +  \sum_{i=
      vt+1}^{vt+h}\big(L_i^\zeta(\eta_{vt}^{\zeta}(v)) -
    L_i^\zeta(\overline{\eta}(v))\big).
  \end{equation}
  To bound the last summand on the right-hand side, we again recall that the family
  $\big(L_i^\zeta(\cdot): i\in\R,0\ge \zeta(\cdot) \ge \ei - \es\big)$ is
  bounded and uniformly equicontinuous on $\triangle$. Therefore, by
  Lemma~\ref{le:concEtaEmpLT_tight}, we have
  \begin{equation*}
    \p\Big( \sup_{\substack{\ln \lfloor t\rfloor\leq |h|\leq \lceil
          t\rceil\varepsilon^*(t),\\ v\in V}} \Big| \frac{1}{h}\sum_{i=vt+1}^{ vt+h}
    \big(L_i^\zeta(\eta_{vt}^{\zeta}(v)) - L_i^\zeta(\overline{\eta}(v))\big)\Big|
    >\frac{\delta}{2},\ vt\geq \Cr{constHn}\ \forall v\in V \Big) \leq c t^{-q}
  \end{equation*}
  for $t$ large enough. It thus suffices to bound the first summand in
  \eqref{eq:dec},
  i.e.~to show that there exists $C_0=C_0(\varepsilon,q)>0$ such that for
  all $t$ large enough we have
  \begin{equation}
    \label{eq:uniform_bound_Hoeffding}
    \p\Big( \sup_{\substack{C_0\ln \lfloor t\rfloor\leq |h|\leq \lceil
        t\rceil\varepsilon^*(t),\\ v\in V }} \Big| L(\overline{\eta}(v)) -
    \frac{1}{h}\sum_{i=vt+1}^{ vt+h} L_i^\zeta(\overline{\eta}(v)) \Big|
    >\frac{\delta}{2}\Big) \leq c t^{-q}.
  \end{equation}

  Hence, for every $h$ we write
  $hL(\eta) - \sum_{i=vt+1}^{ vt+h} L_i^\zeta(\eta)
  = \sum_{i=1}^{\lfloor h\rfloor + 2} \widetilde{L}_i^{\zeta,h,v}(\eta)$,
  where $(\widetilde{L}_i^{\zeta,h,v}(\eta))_{i=1}^{\lfloor |h|\rfloor+2}$
  is a sequence of centered random variables,  which are $\p$-a.s.\
  uniformly bounded in $v\in V$, $h\in\R$, $t\in\R$ and $\eta\in\triangle$, as
  well as fulfill the mixing condition from
  \cite[{Lemma~\ref{LD-le:expMixLi}}]{DreSchmi19}. Thus, we can apply
  Lemma~\ref{lem:rioDepExp_tight} to show that there exist constants $C>0$ and
  $C_0(\varepsilon,q)>0$, such that for all $v\in V$ and all $h$
  fulfilling $|h|\geq C_0\ln t$ we have
  \begin{equation*}
    \p\Big( \Big| L(\eta) - \frac{1}{h}\sum_{i=vt+1}^{ vt+h}
      L_i^\zeta(\eta(v)) \Big| >\frac{\delta}{2}\Big) \leq \sqrt{e}\exp\Big\{
      -\frac{1}{2C\lfloor h \rfloor}\Big( \frac{|h|\varepsilon}{2} \Big)^2  \Big\}
    \leq ct^{-q-3}
  \end{equation*}
  for all $t$ large enough.

  To get the  ‘‘uniform bound'' from
  \eqref{eq:uniform_bound_Hoeffding}, we first show it on the grid
  $V_n:=(\frac{1}{n}\Z) \cap V$
  and
  $C_n^{(t)}:=(\frac{1}{n}\Z)\cap[\ln \lfloor t\rfloor,\lceil t\rceil\varepsilon^*(t)]$,
  $n\in\N$. Indeed, because $|V_n|\leq (\text{diam}(V)+1)n$ and
  $|C_n^{(t)}|\leq nt$, we get
  \begin{equation}
    \label{eq:grid}
    \p\Big( \sup_{{|h| \in C^{(t)}_{\lfloor t\rfloor},\ v\in V_{\lfloor
            t\rfloor} }} \Big| L(\overline{\eta}(v)) -
      \frac{1}{h}\sum_{i=vt+1}^{ vt+h} L_i^\zeta(\overline{\eta}(v))
      \Big| >\frac{\delta}{2}\Big) \leq \text{diam}(V)C\cdot t^{-q}
  \end{equation}
  for all $t$ large enough. To control all $v\in V$ and
  $|h|\in[\ln \lfloor t\rfloor, \lceil t\rceil\varepsilon^*(t)]$, we note
  that for all $s\geq 0$ we have
  \begin{align*}
    &\ln E_{vt+\frac{k}{n}+s}\big[ e^{\int_0^{H_{vt}}(\zeta(B_s)+\eta)\diff s}
      \big] - \ln E_{vt+\frac{k}{n}}\big[ e^{\int_0^{H_{vt}}(\zeta(B_s)+\eta)\diff s}
      \big]
    \\& \quad=
    \ln E_{vt+\frac{k}{n}+s}\big[
    e^{\int_0^{H_{vt}+\frac{k}{n}}(\zeta(B_s)+\eta)\diff s} \big]
           \in\Big[-s\sqrt{2(\es-\ei+|\eta|)},0\Big]
  \end{align*}
  where the last display is again a consequence of \eqref{eq:BS}. Thus
  all $h$  not on the grid the terms in \eqref{eq:grid} differ at most by
  a term of order $1/t$. A similar statement holds for all $v\in V$ not
  on the grid, because $\eta(\cdot)$ is uniformly Lipschitz continuous on
  $V$ (see \eqref{eq:triangleLip}). Thus the uniform bound in
  \eqref{eq:grid} can be extended to be valid for all $h$ such that
  $C_0\ln \lfloor t\rfloor\leq |h|\leq \lceil t\rceil\varepsilon^*(t)$.
  This completes the proof.
\end{proof}

\subsection{Proof of Theorem~\ref{thm:PAMBdd}}

We can now finally return to our first main result: the boundedness of
the front of \eqref{eq:PAM}. Its proof builds on the perturbation
estimate from Lemma~\ref{le:space_perturb} (b), and is rather straightforward.

\begin{proof}[Proof of Theorem~\ref{thm:PAMBdd}]
  Due to \eqref{eq:VEL}, we can choose a compact interval $V\subset (v_c,\infty)$  such that $v_0$ is in the interior of $V$.
  Observe first that the existence of the Lyapunov exponent for the
  solution of \eqref{eq:PAM} (see Proposition~\ref{prop:lyapunov_tight})
  directly implies that the left front $\mover^{M,-}(t)$ as well as the
  right front $\mover^\varepsilon (t)$ of the solution to \eqref{eq:PAM}
  (as defined in \eqref{eq:def_front_KPP})
  satisfy, for arbitrary initial condition
  $u_0 \in \mathcal I_{\mathrm{PAM}}$ and every $\varepsilon ,M >0$,  $\p\text{-a.s.}$
  \begin{equation}
    \label{eq:first_ord_breakpoint}
    \lim_{t\to\infty}\frac{\mover^{\varepsilon}(t)}{t}  =
    \lim_{t\to\infty}\frac{\mover^{M,-}(t)}{t}  =
    v_0.
  \end{equation}
  In particular,
  $\mover^{\varepsilon}(t)/t \in V$ and $\mover^{M,-}(t)/t\in V$ for $t$ large
  enough, since we assume that
  $v_0$ is in the interior of $V$, and
  $a_t:= \mover^\varepsilon(t) - \mover^{M,-}(t) \in o(t)$.

  By Lemma~\ref{le:space_perturb} (b),
  uniformly in $u_0\in \mathcal I_{\mathrm{PAM}}$ and $v\in V$, for all $t$
  large enough such that $vt+a_t\in V$, we get
  \begin{equation}
    \begin{split}
      \frac{u(t,vt+a_t)}{u(t,vt)}
      &= \prod_{k=1}^{\lfloor\sqrt{t} \rfloor}
      \frac{u(t,vt+k a_t/\lfloor\sqrt{t} \rfloor)}
      {u(t,vt+(k-1)  a_t/\lfloor\sqrt{t} \rfloor)}
      \\&\leq \big(\Cr{const_space_perturbation}
        e^{-a_t/(\Cr{const_space_perturbation}\lfloor\sqrt
          t\rfloor)}\big)^{\lfloor\sqrt{t}\rfloor}
      =e^{-a_t/\Cr{const_space_perturbation}(1-  \frac{ \lfloor\sqrt{t} \rfloor}{ a_t}
          \cdot \Cr{const_space_perturbation}\ln\Cr{const_space_perturbation}  )}.
    \end{split}
    \label{eq:perturb_breakpoint}
  \end{equation}

  Now we have all we need to prove Theorem~\ref{thm:PAMBdd}. Set
  $C_{\varepsilon,M}:=2\Cr{const_space_perturbation}\ln\big(\frac{2M\Cr{const_space_perturbation}}{\varepsilon} \big)$
  and
  $\varepsilon(t):=2t^{-1/2}\Cr{const_space_perturbation}\ln \Cr{const_space_perturbation}$.
  Assume by contradiction that the claim of the theorem does not hold.
  Then there exist $0<\varepsilon\le M$ and a (random) sequence
  $(t_n)_{n\in\N}$ such that $t_n\tend{n}{\infty}\infty$ and
  $a_{t_n}=\mover^{\varepsilon}(t_n)-\mover^{M,-}(t_n)\geq C_{\varepsilon,M}$
  for all $n\in\N$. Recalling that $\mover^\varepsilon(t)/t \in V$,
  we get for all $n$ large enough that
  \begin{equation*}
    \varepsilon=u(t_n,\mover^{\varepsilon}(t_n))
    = u\big(t,\mover^{M,-}(t_n)+a_{t_n}\big)
    \leq u(t_n,\mover^{M,-}(t_n)) \cdot
    \Cr{const_space_perturbation}e^{-a_{t_n}/2\Cr{const_space_perturbation}}
    \leq \varepsilon/2,
  \end{equation*}
  where in the first inequality we used Lemma~\ref{le:space_perturb} (b)
  if $a_{t_n}\leq t_n\varepsilon(t_n)$ and \eqref{eq:perturb_breakpoint}
  if $a_{t_n}> t_n\varepsilon(t_n)$. This is a contradiction. As a
  consequence, we must have
  $0\leq \mover^{\varepsilon}(t)-\mover^{M}(t)\leq C_{\varepsilon,M}$ for
  all $t$ large enough.
  Furthermore, this inequality holds uniformly for all $u_0\in\mathcal{I}_{\text{PAM}}(\delta',C')$, because $\Cr{const_space_perturbation}$ is independent of $u_0\in\mathcal{I}_{\text{PAM}}(\delta',C')$, proving the claim of the theorem.
\end{proof}

\section{Unbounded transition front for randomized F-KPP equation}
\label{sec:proof_unbounded}

In this section we show our main results about the transition front for
the solution to \eqref{eq:KPP}, Theorems~\ref{thm:FKPPunBdd}
and~\ref{thm:non-monotone}.
The proofs are based on the following branching process representation of
the solution.

\begin{proposition}[{\cite[{Proposition~\ref{LD-prop:mckean1}}]{DreSchmi19}}]
  \label{prop:KPPBRM}
  Let $\xi: \R \to [0, \infty)$ be a non-negative bounded  function
  satisfying \eqref{eq:hoelder_cont}, $F$ as in \eqref{eq:standard_cond},
  and let $f: \R \to [0,1]$ be a function which can be
  pointwise approximated by an increasing sequence of continuous functions. Then the function
  \begin{equation*}
    w(t,x) := 1- {\tt E}^\xi_{x} \Big[ \prod_{Y \in N(t)} f(Y_t) \Big]
  \end{equation*}
  solves the equation
  \begin{equation*}
     w_t = \frac{1}{2} w_{xx} + \xi(x) F(w)
  \end{equation*}
  with initial condition $w(0, \cdot) =1- f$.
  In particular,
  \begin{equation} \label{eq:probMaxFKPP}
  w(t,x)= \ttP^\xi_x( N^\leq (t,0)\neq \emptyset )
  \end{equation}
   solves this
  equation with
  $f=\mathds{1}_{(0,\infty)}$, i.e.~$w(0,\cdot)=\mathds{1}_{(-\infty,0]}$.
\end{proposition}

\begin{remark} \label{rem:KPPBRM usage}
  Note that Proposition~\ref{prop:KPPBRM} slightly differs from the
  usual McKean representation in homogeneous branching environment.
  More precisely, for $\xi \equiv c$ being a constant function and
  $w(0,\cdot) = \mathds{1}_{(-\infty,0]}$, the canonical representation is given by
  $w(t,x) = \ttP^c_0(N^{\ge}(t,x) \neq \emptyset)$. This representation
  follows from Proposition~\ref{prop:KPPBRM} using the symmetry
  $\ttP^c_x( N^\leq(t,0)\neq \emptyset )=\ttP^c_0(N^\geq(t,x)\neq \emptyset)$
  which is a consequence of the reflection symmetry of the Brownian
  motion and the homogeneity of the environment. However, this identity
  fails to hold if $\xi$ is non-homogeneous.

\end{remark}

\subsection{The potential}

We start the proof of Theorem~\ref{thm:FKPPunBdd} by constructing a
suitable potential $\xi $, for which we then show the unboundedness of
the transition front of the solution to \eqref{eq:KPP}. We fix two
positive finite constants  $\es $ and $\ei$ such that
\begin{align}
  \label{eq:relation_es_ei}
 \frac{\es}{\ei} > 2.
\end{align}
We further let $\delta_1,\delta_2\in(0,1)$ be small positive constants,
which will be fixed at the end of the proof of
Lemma~\ref{lem:LDestimates}, see the paragraph below \eqref{eq:tlb}.

It is an
interesting open question whether the condition \eqref{eq:relation_es_ei}
is necessary for the unboundedness of the front. We could not improve it using
the methods of this paper, see in particular after \eqref{eq:ABmax} where
the condition \eqref{eq:relation_es_ei} is crucially
needed.

Let furthermore $\chi:[0,\infty) \to [0,1]$ be a continuous non-increasing
function with $\chi(x)=1$ for $x\le 1$ and  $\chi(x)=0$ for $x\ge  2$,
and let  $\omega = (\omega^i)_{i\in \mathbb Z}$ be a Poisson point
process on $\mathbb R$ with intensity $1$ constructed on
$(\Omega , \mathcal F, \mathbb P)$. We then define our potential via
\begin{equation}
  \label{eq:def_xi}
  \xi(x):=\ei + (\es-\ei)\cdot\sup\{ \chi(|x-\omega^i|):i\in \mathbb Z\}.
\end{equation}

Observe that the map $x\mapsto\xi(x)$ is a continuous function,
$\xi(x)\in[\ei,\es]$ for all $x\in\R$, $\xi(x)=\ei$ if $|x-\omega^i|>2$
for all $i,$ and $\xi(x)=\es$ if there exists $\omega^i$ such that
$|x-\omega^i|\leq 1$. Also, using the properties of the Poisson point process,
$\xi$ fulfills \eqref{eq:POT}, \eqref{eq:STAT} and \eqref{eq:MIX}. See
Figure~\ref{fig:pot} for an illustration of this potential.

\begin{figure}[ht]
  \centering \includegraphics[width=0.9\linewidth]{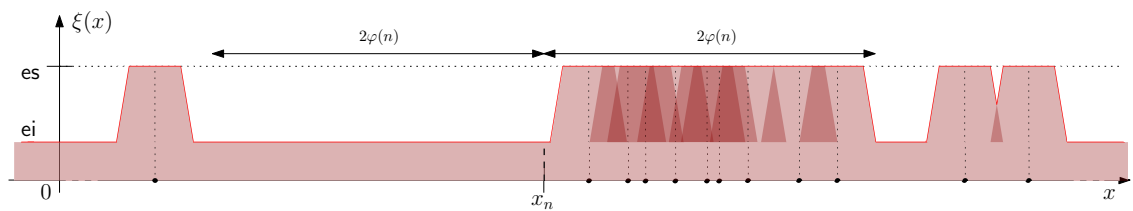}
  \caption{Realization of a potential $\xi$ (top red line) fulfilling
    \eqref{eq:xn_properties} with $\varphi(n)=c_0\ln n$. Here we chose
    $\chi(x)=((3-2x)\wedge 1) \vee 0$.}
  \label{fig:pot}
\end{figure}

The crucial property of this potential is that it has long stretches
where it equals $\ei$ that are adjacent to comparably long stretches
where it equals $\es$, as is proved in the next lemma.

\begin{lemma}
  \label{lem:pot}
  There is a constant $c_0>0$ such that
  $\mathbb P$-a.s.~there exists a (random) increasing sequence
  $(x_n)_{n\in\N}$ of reals tending to infinity, such that
  \begin{align} \label{eq:xn_properties}
    \begin{split}
      \xi(x) &= \ei \quad \forall x \in [x_n- 2 c_0\ln n, x_n],\\
      \xi(x) &= \es \quad \forall x \in [x_n+2,x_n+2 c_0 \ln n-2],
    \end{split}
  \end{align}
  and $\xi(\cdot)$ is non-decreasing on $[x_n-2c_0\ln n, x_n+2c_0\ln n-2]$.
  Moreover, $\mathbb P$-a.s.,
  \begin{align}
    \label{eq:xn_propertiesGrowth}
    1\le \liminf_{n\to\infty} n^{-1 }x_n \le  \limsup_{n\to\infty} n^{-1}x_n \le 2 .
  \end{align}
\end{lemma}

\begin{proof}
  The proof is an easy application of the Borel-Cantelli lemma.
  For $k \in \mathbb N$, let $A_{k,n}$ be the event
  \begin{equation*}
    A_{k,n} = \bigg\{\,
      \begin{aligned}
        \omega :{}&\omega \cap [n+(4k -2)c_0\ln n-2,n+4kc_0\ln n+2)  = \emptyset
        \quad \text{and }
        \\&\omega \cap [n+4kc_0 \ln n  +\ell, n+4k c_0 \ln n + \ell +1) \neq \emptyset
        \text{ for all }
          \ell = 2,\dots,\lfloor 2 c_0 \ln n\rfloor-3\,
        \end{aligned}
      \bigg\}.
  \end{equation*}
  Observe that if $A_{k,n}$ occurs, then $\xi $ satisfies
  \eqref{eq:xn_properties} with $x_n = n+4k c_0 \ln n$, and that $A_{k,n}$
  only depends on $\omega  $ in the interval
  $[n+(4k-2)\ln n -2, n+4k \ln n + \lfloor 2 \ln n \rfloor -2)$.
  Therefore,  the events $(A_{k,n})_{k\in \mathbb N}$ are independent.
  Moreover,
  \begin{equation*}
    \mathbb P(A_{k,n}) = e^{- 2c_0\ln n - 4}
    \prod_{\ell=2}^{\lfloor 2 c_0 \ln n\rfloor-3}(1-e^{-1}) \ge \alpha
    ^{-c_0 \ln n} = n^{-c_0 \ln \alpha }
  \end{equation*}
  for some $\alpha >1 $ independent of $c_0$. Therefore, using $1-x \le e^{-x}$,
  \begin{equation*}
    \mathbb P\Big(\bigcap_{k=0}^{n/(4c_0\ln n)-1} A_{k,n}^c\Big)
    \le (1-n^{-c_0 \ln \alpha })^{n/(4c_0\ln n)} \le
    \exp\{ -n ^{1-c_0 \ln \alpha } (4 c_0\ln n)^{-1}\}.
  \end{equation*}
  For $c_0< 1/\ln \alpha $, the right-hand side is summable and thus by
  the Borel-Cantelli lemma, almost surely for $n$ large enough, there
  exists $k\in [0, n/(4c_0 \ln n)-1]$ such that $A_{k,n}$ occurs. This
  implies that $\mathbb P$-a.s.~for $n$ large enough there is
  $x \in [n,2n]$ satisfying \eqref{eq:xn_properties}, completing the proof.
\end{proof}

In the following, if not mentioned otherwise, we will always refer to the
sequence $(x_n)$ as the one the existence of which is provided by
\eqref{lem:pot}.

\subsection{The coupling}
\label{ss:coupling}

In the next step towards a proof of Theorem~\ref{thm:FKPPunBdd}, we
construct a coupling of two BBMREs started in the vicinity of the points
$x_n$ where the potential satisfies the conditions
\eqref{eq:xn_properties} of Lemma~\ref{lem:pot}.

Throughout this section, we assume that the constant $c_0$ and the random
sequence $x_n$ are as in Lemma~\ref{lem:pot}, and write
\begin{equation}
  \label{eq:phiDef}
  \varphi(n)=c_0\ln n.
\end{equation}
In order to emphasize the dependence of the BBMRE on the starting point,
we write $N_x = (N_x(t))_{t\ge 0}$ for the BBMRE started from $x$, that
is for the process whose distribution is $\ttP^\xi_x$.

The content of the next proposition is the coupling alluded to above. Its
statement is slightly more general than needed to show
Theorem~\ref{thm:FKPPunBdd}, since we construct couplings for many
different starting points. This additional control will be useful in the
proof of Theorem~\ref{thm:non-monotone}. Recall that the (possibly small
  but) positive parameter $\delta_1$ is fixed below \eqref{eq:tlb}.

\begin{proposition}
  \label{pro:coupling}
  For every $\varepsilon >0$ there exists
  $\Cl{const_coupling}=\Cr{const_coupling}(\varepsilon)\in (0,\infty)$
  such that for all $n$ large enough,
  $l \in [x_n-5\delta_1 \varphi (n), x_n-4\delta_1 \varphi (n)]$, and
  $r \in [x_n+\delta_1 \varphi (n), x_n+2\delta_1 \varphi (n)],$ there
  exists a coupling $\ttQ_{l,r}^\xi $ of the BBMREs $N_l$ and $N_r$ such
  that
  \begin{equation}
    \label{eq:success}
    \ttQ^\xi_{l,r} (N_l(t) \subset N_r(t) \,\forall t \ge \Cr{const_coupling} \ln n ) \ge
    1-\varepsilon .
  \end{equation}
\end{proposition}
For an illustration of the coupling and an explanation of the strategy to
show that the event in \eqref{eq:success} occurs with high probability,
we refer to Figure~\ref{fig:coupling}.

Before proving Proposition~\ref{pro:coupling}, let us first show that it
implies Theorem~\ref{thm:FKPPunBdd}.
\begin{proof}[Proof of Theorem~\ref{thm:FKPPunBdd}]
  Using the notation from Proposition~\ref{pro:coupling} we set
  \begin{equation*}
    t_n:=\inf\{t\geq0: w(t,x_n-4\delta_1\varphi(n))=\delta\}.
  \end{equation*}
  Note that $t_n \geq \Cr{const_coupling}\ln n$ for all $n$ large enough
  (using $x_n\geq n$ and the fact that the front moves linearly, see
    Proposition~\ref{prop:lyapunov_tight}).
  By
  \eqref{eq:xn_propertiesGrowth} and \eqref{eq:phiDef} we get
  $\varphi\in\Omega(\ln n)$,  $x_n,t_n\to\infty$,
  $(x_n)_{n\in\N}\in\mathcal{O}(n)$ and it remains to show
  \eqref{eq:unbddFront}.  Let us abbreviate  $l:=x_n-4\delta_1\varphi(n)$
  and $r:=x_n+2\delta_1\varphi(n)$. By definition of the coupling
  $\ttQ^\xi_{l,r}$ and the representation
  $w(t,x)=\ttP_{x}^\xi\big( N^\leq(t,0)\neq\emptyset \big)$ of the
  solution to \eqref{eq:KPP} (see Proposition~\ref{prop:KPPBRM}), we have
  for all $n$ large enough that
  \begin{align*}
    \delta = w(t_n,x_n-4\delta_1\varphi(n))
    &= \ttP_{l}^\xi\big( N^\leq(t_n,0)\neq\emptyset \big)
    = \ttQ^\xi_{l,r} \big( N_l^\leq(t_n,0)\neq\emptyset \big) \\
    &\leq \ttQ^\xi_{l,r} \big( N_l^\leq(t_n,0)\neq\emptyset,
      N_l(t) \subset N_r(t) \,\forall t \ge \Cr{const_coupling} \ln n \big) + \varepsilon \\
    &\leq \ttQ^\xi_{l,r} \big( N_r^\leq(t_n,0)\neq\emptyset \big) + \varepsilon
    = \ttP_{r}^\xi\big( N^\leq(t_n,0)\neq\emptyset \big) + \varepsilon \\
    &=w(t_n,x_n+2\delta_1\varphi(n))+\varepsilon,
  \end{align*}
  where we used  \eqref{eq:success} in the first inequality. Adapting the
  notation to that of the statement, we can conclude.
\end{proof}

\begin{proof}[Proof of Proposition \ref{pro:coupling}]
  To construct the coupling, we endow every particle in $N_l$ and $N_r$
  at every time with a type. The type of the particle does not influence
  its dynamics within $N_l$ or $N_r$, but rather helps to encode the
  dependence between $N_l$ and $N_r$ under $\ttQ^\xi_{l,r}$. At any given
  time, every particle in $N_l$ can have either of the types
  \emph{l-mirrored}, \emph{l-coupled}, or \emph{bad}. Similarly, every
  particle in $N_r$ can have either of the types \emph{r-mirrored},
  \emph{r-coupled}, or \emph{free}. We denote $\LM(t), \LC(t), \Bad(t)$
  and $\RM(t)$, $\RC(t)$ and $\Free(t)$ the sets of particles with those
  respective types at time $t$. A particle is given a type when it is
  created, and its type can change only if it branches, meets another
  particle or hits some special point in space, as we will describe
  later. The assignment of the type is a right-continuous function in
  times, in the sense that if, e.g., a particle $Y$ changes its type
  from l-mirrored to bad at time $t$, then $Y\in \Bad(t)$ and
  $Y\in \LM(t-)$.

  In addition, under the coupling, at every time $t\ge 0$, there are bijections
  $\mu_t:\LM(t) \to \RM(t)$ and $\gamma_t:\LC(t) \to \RC(t)$. The
  bijections $\mu_t $ ``mirror'' the positions of the particles:
  \begin{equation}
    \label{eq:mirroring}
    \text{If $Y\in \LM(t)$ and $Y' = \mu_t(Y)\in \RM(t)$, then
      $m - Y_t = Y'_t - m$,}
  \end{equation}
  where $m$ is the midpoint of the segment $(l,r)$,
  \begin{equation*}
    m := \frac 12 (l+ r) \in [x_n - 2 \delta_1 \varphi (n), x_n - \delta_1\varphi (n)].
  \end{equation*}
  On the other hand, coupled particles are at the same position:
  \begin{equation}
    \label{eq:coupling}
    \text{If $Y \in \LC(t)$ and $Y' = \gamma_t (Y) \in \RC(t)$, then
      $Y_t = Y'_t$.}
  \end{equation}
  As time evolves, the bijections $\mu_t$ and $\gamma_t$  naturally follow
  the particles. That is, for the mirrored particles, if $Y \in  \LM(t) \cap \LM(t')$,
  $Y' \in \RM(t)\cap \RM(t')$ and $Y'= \mu_t(Y)$, then also
  $Y'= \mu_{t'}(Y)$, and similarly for the coupled particles.

  We set
  \begin{equation}
    \label{eq:outer_bounds}
    \lbound := x_n - \varphi(n)  \qquad \text{and}\qquad \rbound := 2m - \lbound .
  \end{equation}
  It will turn out that under the coupling constructed below, the
  l-mirrored particles will always be in the interval $(\lbound,m)$, that
  is $\{Y_t: Y\in \LM(t)\}\subset (\lbound,m)$, see \ref{it:hit_m} and
  \ref{it:hit_L} below. As a consequence of \eqref{eq:mirroring} and
  \eqref{eq:outer_bounds}, we then have
  $\{Y_t: Y\in \RM(t)\}\subset (m,\rbound)$. In particular, in
  combination with \eqref{eq:xn_properties}, we infer that the potential
  is always larger at the position of an r-mirrored particle than at the
  position of the corresponding l-mirrored particle:
  \begin{equation}
    \label{eq:ordering}
    \text{If $Y\in \LM(t)$ and $Y'= \mu_t (Y)$, then
      $\xi (Y_t) \le \xi (Y'_t)$.}
  \end{equation}

  We can now describe the dynamics of $N_l$, $N_r$ and of the types under
  the coupling $\ttQ_{l,r}^\xi$. At time $0$, there is one (l-mirrored)
  particle at position $l$ in $N_l$ and one (r-mirrored) particle at
  position $r$ in $N_r;$ this determines the bijection $\mu_0$ uniquely.
  Every particle in $N_l$ (resp.~$N_r$) performs Brownian motion,
  independently of the other particles in $N_l$ (resp.~$N_r$). The
  corresponding mirrored and coupled particles are required to satisfy
  \eqref{eq:mirroring} and \eqref{eq:coupling} respectively, which is
  possible, since the law of Brownian motion is invariant by reflection;
  besides these two conditions the motion of particles in $N_l$ is
  independent of the motion of particles in $N_r$.

  \begin{figure}[t]
    \centering
    \includegraphics[width=0.82\linewidth]{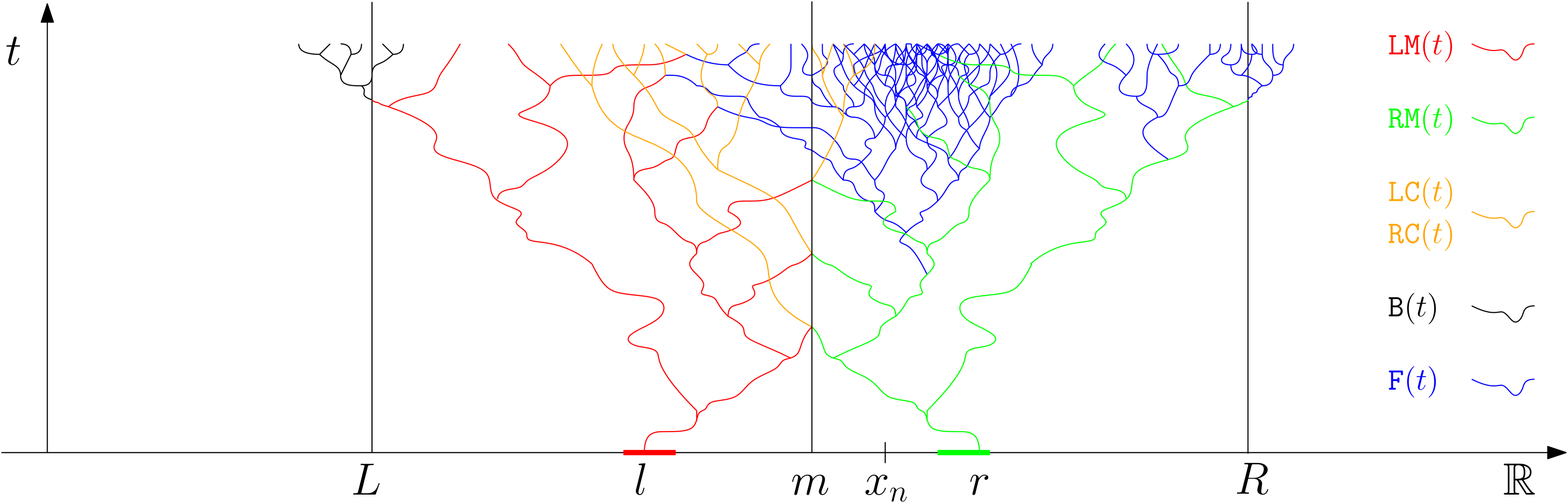}
    \caption{An illustration of the coupling mechanism. l-mirrored
      particles are illustrated in red, r-mirrored particles in green,
      while l- and r-coupled particles are illustrated in orange. Free
      particles are blue and bad particles are black. The fat red (resp.
        green) line on the $\R$-axis denotes the set
      $[x_n-5\delta_1\varphi(n),x_n-4\delta_1\varphi(n)]$ (resp.
        $[x_n+\delta_1\varphi(n),x_n+2\delta_1\varphi(n)]$). Note that
      $x_n$ is nearer to the green domain, forcing a particle $Y\in N_l$
      to go a long way to reach high branching-potential.  The event in
      \eqref{eq:success} occurs, if at time $t=c\ln n$, all l-mirrored
      particles (red) are already turned into l-coupled ones (orange) and
      no l-mirrored particles have crossed $\lbound$ yet. But then there
      will  be no bad particles (black) either, which already implies the
      event in \eqref{eq:success}.}
    \label{fig:coupling}
  \end{figure}

  The branching events occur according to the following rules.
  \begin{enumerate}
    \item At time $t$, every $Y\in N_l$ branches
    with rate $\xi (Y_t)$. It is replaced by $k$ new
    particles,
    with probability $p_k$, independently of remaining randomness. The
    type of the new particles is the same as of $Y$.

    If a particle $Y$ is l-mirrored (resp.~l-coupled), $Y\in \LM(t-)$
    (resp.~$Y\in\LC(t-)$) before time $t,$ then the corresponding r-mirrored
    particle $Y'=\mu_{t-}(Y)$ (resp.~r-coupled particle,
      $Y' = \gamma_{t-}(Y)$)  branches as well. It is replaced by the
    same number $k$ of particles. The newly created particles are set to
    be r-mirrored (resp.~r-coupled) and the bijection $\mu_t$ (resp.
      $\gamma_t$) is a natural extension of $\mu_{t-}$ (resp.~
      $\gamma_{t-}$) to the newly created particles.

    \item At time $t$, every r-mirrored particle $Y'\in \RM(t-)$ (mirrored
      with $Y= \mu_{t-}^{-1}(Y')$)
      branches with rate
    $\xi (Y'_t)-\xi(2m - Y'_t) = \xi (Y'_t) - \xi (Y_t),$
    in addition to the branching occurring in (a). This rate is
    non-negative due to \eqref{eq:mirroring} and \eqref{eq:ordering}.
    It is replaced by $k$ new particles,
    with probability $p_k$, independently of everything else.
    One of the newly created particles, say $Z'$, is set to be
    r-mirrored, and  we set $\mu_{t}(Y) := Z'$. The type of the remaining
    newly created particles is free.

    \item At time $t$, every free particle $Y'\in \Free(t)$ branches with rate
    $\xi (Y'_t)$. It is replaced by $k$ new particles, with
    probability $p_k$, independently of everything else. The type of
    the new particles is free.
  \end{enumerate}

  It can be easily checked that, as a result of the rules (a)--(c), every
  $Y'\in N_r$ branches with rate $\xi (Y'_t)$ at time $t$, as it should.

  Finally, the particles can change their type if one of the following events
  occur:
  \begin{enumerate}[(A)]
    \item \label{it:hit_m} If an l-mirrored particle hits $m$, that is $Y\in \LM(t-)$ and
    $Y_t = m$, then, by consequence of
    \eqref{eq:mirroring}, the corresponding particle $Y'=\mu_{t-}(Y)$
    satisfies $Y'_t= m$ as well. We thus change the types of $Y$ and $Y'$
    to l-coupled and r-coupled, respectively, and define
    $\gamma_{t}(Y) :=  Y'$.

    \item If an l-mirrored particle $Y\in \LM(t-)$ meets a free particle at
    time $t$, that is there is $Z'\in \Free(t-)$ with $Z'_t= Y_t$, then
    we change the types of $Y$ and $Z'$
    to l-coupled and r-coupled, respectively, and define
    with $\gamma_{t}(Y) :=  Z'$. The type of the r-mirrored particle
    $Y'=\mu_{t-}(Y)$ that was
    mirrored with $Y$ is changed to free.

    \item \label{it:hit_L} If an l-mirrored particle hits $\lbound$, that is $Y\in \LM(t-)$ and
    $Y_t=\lbound$, then the type of  $Y$ is changed to bad, and the type of
    the corresponding r-mirrored particle $Y' = \mu_{t-}(Y)$ is changed to
    free.
  \end{enumerate}

  To show that the coupling succeeds, i.e.\ that \eqref{eq:success}
  holds, it is sufficient to show that with probability at least $1-\varepsilon,$
  there are no l-mirrored and bad particles after time $\Cr{const_coupling}\log n$.
  In this vein, we define two good events:
  \begin{equation}
    \mathcal G_1(t) := \{N_l^\le (s,\lbound) = \emptyset \, \forall s\le t \},
  \end{equation}
  i.e.,~on $\mathcal G_1(t)$ no particle from $N_l$ enters $(-\infty,\lbound)$
  before time $t$, and
  \begin{equation}
  \mathcal G_2(t) := \big\{ N_r^{\le}(t,\lbound)\neq \emptyset\big\};
  \end{equation}
  i.e.,~there is a (necessarily free, if $\mathcal G_1(t)$ occurs as
    well) particle to the left of $\lbound$ at time $t$. We now need the
  following lemma which ensures that we can find $t$ such that those
  events are typical.

  \begin{lemma}
    \label{lem:LDestimates}
    For any $\varepsilon >0$ there exists $t'< 1$ such that for all $n$
    large enough, with $t = t' \varphi (n)/\sqrt {2\ei}$,
    \begin{equation}
      \ttQ_{l,r}^\xi \big(\mathcal G_1(t) \cap \mathcal G_2(t) \big) \ge
      1-\varepsilon .
    \end{equation}
  \end{lemma}

  We postpone the proof of this lemma and complete the proof of
  Proposition~\ref{pro:coupling} first. Let $t$ be as in
  Lemma~\ref{lem:LDestimates}. We claim that
  \begin{equation}
    \label{eq:subset}
    \{N_l(t) \subset N_r(t) \} \supset \mathcal G_1(t) \cap \mathcal
    G_2(t).
  \end{equation}
  If we show this, then the claim of Proposition~\ref{pro:coupling}
  follows with  $\Cr{const_coupling}  = t/\ln n = t'c_0 /\sqrt {2 \ei}$.

  To prove \eqref{eq:subset}, recall first that bad particles can only be created
  if an l-mirrored particle hits $\lbound$. As a consequence,
  \begin{equation} \label{eq:noBadG1}
  \text{on $\mathcal G_1(t)$ there cannot be any bad particles at time $t$.}
  \end{equation}
  Next, we show that
  \begin{equation} \label{eq:nolMirG1G2}
  \text{on
  $\mathcal G_1(t) \cap \mathcal G_2(t)$ there are no l-mirrored
  particles at time $t$}
  \end{equation}
  either. To this end define
  $\rmin(t) = \inf\{Y'_t: Y\in \Free(t)\}$ to be the position of the
  leftmost free particle, and $\lmax(t) = \sup\{Y_t:Y\in \LM(t)\} $ to be
  the position of the rightmost l-mirrored particle, with the convention
    $\inf \emptyset = +\infty$, $\sup \emptyset = - \infty;$ in the remaining
    cases, a.s., the infimum and supremum are attained, since $\Free(t)$
    and $\LM(t)$ are a.s.~finite sets). Let
  $$
  \tau   := \inf\{t\ge 0:  \lmax(t) > \rmin(t)\}.
  $$
  We claim that $\tau  = \infty$, $\ttQ_{l,r}^\xi $-a.s. Indeed, we first
  note that $\lmax $ and $\rmin$ are right-continuous. In addition, the
  only jumps that $\lmax $ has are downward jumps. They occur a.s.~iff
  the rightmost l-mirrored particle changes its type due to (A)--(C). (If
    one of (A)--(C) occurs, then a.s.~there is only one l-mirrored
    particle at position $\lmax(t)$. At branching events, $\lmax$ is
    unchanged, as l-mirrored particles are created only at positions
    where l-mirrored particles are already present, see (a)). Similarly,
  with the exception of the first jump from $+\infty$, the only jumps
  that the function $\rmin$ has are upwards jumps, occurring a.s.~iff the
  leftmost free particle becomes r-coupled due to (B). Therefore, it
  follows that a.s. $\tau \ge \inf\{t\ge 0:\lmax(t) =\rmin(t)\}$.
  However, the event $\{\exists t \in [0,\infty):\lmax(t) =\rmin(t)\}$
  cannot occur by the construction of the coupling, since if an
  l-mirrored and a free particle meet, then at this instant they become
  l-/r-coupled immediately. Hence, $\tau =\infty$ almost surely, as claimed.

  Assume now that $\mathcal G_1(t) \cap \mathcal G_2(t)$ occurs. At time
  $t$, there is thus a particle from $N_r$ and no particle from $N_l$ to
  the left of $\lbound$. From the construction, this particle is neither
  r-coupled (since on $\mathcal G_1(t)$ there is no corresponding
    l-coupled particle there), nor r-mirrored (as all r-mirrored
    particles are always in $(m,\rbound)$). Therefore, it must be free
  and thus $\rmin(t)<\lbound$. Since $\tau =\infty$ a.s.,
  $\lmax(t)<\lbound$ as well. However, by construction, l-mirrored
  particles are always located in $(\lbound,m)$, and thus
  $\lmax(t) < \lbound$ implies $\lmax(t) = -\infty$, that is
  $\LM(t) = \emptyset,$ establishing \eqref{eq:nolMirG1G2}.

  All in all, from the above it follows that on
  $\mathcal G_1(t)\cap\mathcal  G_2(t),$ \eqref{eq:noBadG1} as well as
  \eqref{eq:nolMirG1G2} hold true, i.e.,  there do not exist any
  l-mirrored or bad particles at time $t$. Hence, all particles in
  $N_l(t)$ are necessarily l-coupled, which proves \eqref{eq:subset}.
  This completes the proof of Proposition~\ref{pro:coupling}.
\end{proof}

It remains to show Lemma~\ref{lem:LDestimates}.

\begin{proof}[Proof of Lemma~\ref{lem:LDestimates}]
  We first estimate the probability of $\mathcal G_1(t)$ as a function of
  $ t \in [0, \varphi (n)/\sqrt {2 \ei}] $. To this end we write $\mathcal N(t)$ for the
  number of particles from $N_l$ that hit $\lbound$ before $t;$ here, we only count
    the first hit of $\lbound$ by any particle. That is, we disregard possible
    successive hits of $\lbound$ by the same particle, and also the fact that this
    particle could branch between the hitting of $\lbound$ and the time $t$, and
    thus produce more particles at time $t$ that hit $\lbound$. The expectation
  of $\mathcal N(t)$ can be written as
  \begin{equation}
    \ttE_l^{\xi}[\mathcal N(t)]
    = E_l \Big[ e^{\int_{0}^{H_{\lbound}} \xi (X_s) \d s}; H_{\lbound} < t \Big]
    \le
    E_l \Big[ e^{\int_{0}^{H_{\lbound}} \tilde \xi (X_s) \d s}; H_{\lbound} < t \Big],
  \end{equation}
  where the potential $\tilde\xi $ is given by $\tilde \xi (x) = \es$ if
  $x\ge x_n$, and $\tilde \xi (x) = \ei$ if $x< x_n$. To estimate the
  right-hand side, note that there are two possible scenarios for a
  particle to hit $\lbound$. Either, it stays all the time in the interval
  $(\lbound,x_n)$ where the potential equals $\ei$ and hits $\lbound$ (i.e., it
  displaces by altogether at least $l-\lbound \ge (1- 5\delta_1) \varphi (n)$). Or, it
  spends some $s$ units of time in the interval $[x_n,\infty)$, where the
  potential is $\es$, but then it should displace by at least
  $(x_n-l)+ (x_n - \lbound) \geq  (1+4\delta_1) \varphi (n)$ in $t-s$ units of
  time. Ignoring prefactors which are sub-exponential in
  $\varphi (n)$ and using standard Gaussian tail bounds, we thus arrive at the
  following upper bound:
  \begin{equation}
    \begin{split}
      \label{eq:Gone}
      \ttE_l^{\xi}[\mathcal N(t)] &\lesssim
      \exp\Big\{ t \ei  - \frac {(1-5\delta_1)^2 \varphi (n)^2}{2t}  \Big\}
      + \sup_{s\le t}
      \exp\Big\{ (t-s) \ei + s \es  - \frac {(1+4\delta_1)^2 \varphi
          (n)^2}{2(t-s)}  \Big\}
      \\ &=
      \exp\Big\{ \sigma (n)
        \Big(t' - \frac{(1-5 \delta_1)^2}{t'}\Big)\Big\}
      +
      \sup_{s'<t'}
      \exp\Big\{  \sigma(n)
        \Big(t' + s'\,\frac{\es -\ei}{\ei}
          -\frac{(1+4 \delta_1)^2}{t'-s'}\Big)\Big\},
    \end{split}
  \end{equation}
  where we introduced
  \begin{equation}
    \label{eqn:defsigma}
    \sigma (n) = \varphi (n)\sqrt {\frac \ei 2} \qquad \text{and} \qquad
    t' = \frac{t\, \ei}{  \sigma (n)}
  \end{equation}
  in order to put the various terms on the same scale. Using Markov's
  inequality, to show that $\ttP_l^\xi(\mathcal G_1(t)^c)\to 0,$ it is
  sufficient to show that both summands on the right-hand side of
  \eqref{eq:Gone} tend to $0.$ For this to be the case for the first one,
  it is sufficient to require
  \begin{equation}
    \label{eq:tub}
    t'<(1-5\delta_1).
  \end{equation}
  Before dealing with the second term (which we will do below
    \eqref{eq:Gtwo}), we turn our attention to the event $\mathcal G_2(t)$.

  To control the probability of  the event $\mathcal G_2$, we need two claims.

  \begin{claim}
 	\label{cl:many_part}
    For every $\varepsilon >0$ there exists $t_0<\infty$ such that for all $n$ large enough,
    \begin{equation}
    \label{eq:many_part}
      \ttP_r^{\xi}\Big(\big |\big \{Y\in N_r(t): Y_t\in [x_n+\delta_1 \varphi (n),
            x_n + 2 \delta_1 \varphi (n) ] \big \} \big  | \ge
        e^{(1-\delta_2)\es \, t}\Big) \ge 1- \varepsilon/2  ,
      \qquad \text{for  all $t\ge t_0$.}
    \end{equation}
  \end{claim}

  In order not to hinder the flow of reading, we postpone the proof
  of Claim~\ref{cl:many_part} to the end of the proof of
  Lemma~\ref{lem:LDestimates}.

  \begin{claim}
    Let $t=t'\varphi(n)/\sqrt{2 \ei}$ with $t'<1$ and $\eta >0$. Then for
    every $y\in[x_n+\delta_1 \varphi (n), x_n + 2 \delta_1 \varphi (n) ]$
    and all $n$ large enough
    \begin{equation}
      \ttP_y^\xi(N^\le(t,\lbound) \neq \emptyset)
      \ge \exp \Big\{\sigma (n)\Big(t'
          - \frac{(1+2\delta_1)^2}{t'} - \eta \Big)\Big\}.
    \end{equation}
  \end{claim}
  \begin{proof}
    Obviously
    $\ttP_y^\xi(N^\le(t,\lbound) \neq \emptyset) \ge \ttP_y^\ei(N^\le(t,\lbound) \neq \emptyset)
    \ge \ttP_{x_n + 2 \delta_1 \varphi (n)}^\ei(N^\le(t,\lbound) \neq \emptyset)$.
    Moreover, by the large deviation lower bound from \cite[Thm.\ 1]{CR88}, for
    every $v>\sqrt{2\ei}$ and $\eta >0$, if $t$ is sufficiently large,
    then
    \begin{equation*}
      \ttP_0^\ei(N^\le(t,-vt) \neq \emptyset)
      \ge \exp\{t(\ei - v^2/2 - \eta )\}.
    \end{equation*}
    Using this estimate with
    $v=(x_n+2\delta_1 \varphi (n)-\lbound)/t = (1+2\delta_1) \varphi (n) / t = (1+2\delta_1) \sqrt {2\ei}/t'> \sqrt{2\ei}$,
    and by rewriting it using the notation introduced in
    \eqref{eqn:defsigma}, the claim follows.
  \end{proof}
  Using these two claims, we have that for any $0< s'<t'<1$ as well as for
  $t = t' \varphi (n)/\sqrt{2 \ei}$ and $s = s'\varphi(n)/\sqrt{2\ei},$ that
  \begin{equation*}
    \begin{split}
      \ttP_r^\xi (\mathcal G_2(t)^c) &\le
        \ttP_r^{\xi}\Big(|\{Y\in N_r(s): Y_s\in [x_n+\delta_1 \varphi (n),
              x_n + 2 \delta_1 \varphi (n) ]\} \le e^{(1-\delta_2)\es \, s}\Big)
        \\&\quad +\ttP_r^\xi\Big(\mathcal G_2(t)^c\, \big|\, |\{Y\in N_r(s): Y_s\in [x_n+\delta_1 \varphi (n),
              x_n + 2 \delta_1 \varphi (n) ]\}| \ge e^{(1-\delta_2) \es \, s}\Big)
        \\&\le \frac \varepsilon 2
        + \Big(1-\exp\Big\{\sigma (n) \Big(t'-s' - \frac{(1+2\delta_1)^2}{t'-s'} -
              \eta\Big)\Big\}\Big)^{\exp\{(1-\delta_2)\es \, s\}}.
      \end{split}
  \end{equation*}
  The second summand on the right-hand side converges to $0$ as
  $n\to \infty$ if
  \begin{equation}
    \begin{split}
      \label{eq:Gtwo}
      &\exp\Big\{\sigma (n) \Big(t'-s' - \frac{(1+2\delta_1)^2}{t'-s'} -
          \eta\Big)\Big\} \cdot
      \exp\{(1-\delta_2)\es \, s\}
      \\&=
      \exp\Big\{\sigma (n) \Big(t'+s'\frac{\es(1-\delta_2)-\ei}\ei
          - \frac{(1+2\delta_1)^2}{t'-s'} - \eta\Big)\Big\}
      \to \infty.
    \end{split}
  \end{equation}

  The factors in the exponents of \eqref{eq:Gone} and \eqref{eq:Gtwo} are
  both of the form $t' + A s' - B/(t'-s')$ such that (for $\delta_2 > 0$
    small) $A>0$ and $B>1$. For $A$, $B$ and $t'$, fixed, this function
  is maximized for $s\in [0,t']$ by
  \begin{equation}
    s' =
    \begin{cases}
      t' - \sqrt{B/A},&\\
      0,&
    \end{cases}
    \text{ with a maximum value of }
    \begin{cases}
      (1+A)t'-2\sqrt{AB},\quad&\text{if $t'> \sqrt{B/A}$,}\\
      t' - B/t',&\text{otherwise.}
    \end{cases}
    \label{eq:ABmax}
  \end{equation}
  Ignoring for a moment the constants $\delta_2$ and $\eta $, we write
  $A=(\es - \ei)/\ei$, $B_1 = (1+4 \delta_1)^2$, and $B_2 = (1+2\delta_1)^2$.
  Observe that $A>1$ by \eqref{eq:relation_es_ei}. In order to satisfy
  \eqref{eq:Gtwo} and let \eqref{eq:Gone} tend to 0, we must fix $t'$
  and $\delta_1$ so that \eqref{eq:tub} holds, and at the same time
  \begin{gather}
    \label{eq:negsup}
    \sup_{0<s'<t'} t' + s' A - B_1 /(t'-s') <0,\\
    \label{eq:possup}
    \sup_{0<s'<t'} t' + s' A - B_2 /(t'-s') >0.
  \end{gather}
  Since $B_2>1$ and $t'<1$, the analysis in \eqref{eq:ABmax} implies that
  the supremum in \eqref{eq:possup} can be positive only if
  \begin{equation}
    t'>\max\bigg(\sqrt{\frac{B_2}A},\frac{2\sqrt{AB_2}}{1+A}\bigg) =
    \frac{2 \sqrt {AB_2}}{1+A},
    \label{eq:tlb}
  \end{equation}
  where to obtain the equality we used the fact that $A>1$. We thus fix
  $\delta_1>0$ small enough so that $1-5\delta_1 > 2 \sqrt{AB_2}/(1+A)$
  and \eqref{eq:tub} as well as \eqref{eq:tlb} can be both satisfied;
  this is possible only if $A>1$ which is true by assumption. We then fix
  $t'$ satisfying \eqref{eq:tub} and \eqref{eq:tlb}, so that the supremum
  in \eqref{eq:possup} is positive (this is by construction), but small
  enough, so that the supremum in \eqref{eq:negsup} is negative; this is
  possible since $B_1>B_2$. Finally, we fix $\delta_2>0$, $\eta>0 $ so
  that the validity of the established inequalities is not modified. With
  this choice of constants, \eqref{eq:Gtwo} holds and the right-hand side
  of \eqref{eq:Gone} tends to $0$, as required. Hence, for
  $t = t' \varphi (n)/\sqrt{2\ei}$ we have
  $\ttQ_{r,l}^\xi  (\mathcal G_1(t)^c \cup \mathcal G_2(t)^c)
  \le \ttP_{l}^\xi  (\mathcal G_1(t)^c) + \ttP_{r}^\xi  ( \mathcal G_2(t)^c) \to 0$
  as $n\to \infty$. This completes the proof.
\end{proof}

It remains to prove Claim~\ref{cl:many_part}.

\begin{proof}[Proof of Claim~\ref{cl:many_part}]
  The proof follows by a comparison with branching processes split into two phases. For
  the first phase we recall that by
  Lemma~\cite[Lemma~\ref{LD-le:amplification}]{DreSchmi19} there exist
  $\kappa>1$ and $t_1<\infty$ such that, $\mathbb P$-a.s.,
  \begin{equation}
    \label{eq:ampli}
    \sup_{x\in \mathbb R} \mathtt P^\xi_x \big (|\{Y\in N(t): Y_t \in
        [x-1,x+1]\}|\le \kappa^t \big) \le \kappa^{-t}\qquad
    \text{for all } t\ge t_1.
  \end{equation}
  For the second phase we need few preparatory steps. We fix $T>0$ such that
  \begin{equation}
    \label{eq:Tfix}
    e^{(1-\frac{\delta_2}2)\es\, T} \le \frac 14 e^{\es\, T}
    \quad\text{and}\quad
    P_0(B_T>1)\ge \frac 7{16}.
  \end{equation}
  We further fix $K_1>1$ large enough so that
  \begin{equation}
    \label{eq:K1fix}
    \inf_{x\in [-K_1-1,K_1+1]} P_x\big(B_T\in [-K_1,K_1]\big) \ge \frac
    38,
  \end{equation}
  which is possible due to the second part of \eqref{eq:Tfix}.
  Finally, we fix $K_2>K_1$ large enough so that
  \begin{equation}
    \label{eq:K2fix}
    \sup_{x\in [-K_1-1,K_1+1]} P_x\big( B_s\notin [-K_2,K_2] \, \text{ some } s\in
      [0,T]\big) \le \frac
    1{16},
  \end{equation}
  so \eqref{eq:K1fix} in combination with \eqref{eq:K2fix} entail that
  \begin{equation}
    \label{eq:Kfix}
    \inf_{x\in [-K_1-1,K_1+1]} P_x\big(B_T\in [-K_1,K_1], B_s \in
      [-K_2,K_2]\,\forall s\le T\big)  \ge
    \frac 5{16}.
  \end{equation}
  Next, assume that $n$ is large enough, so that
  $\delta_1 \varphi (n) > K_2/2$, and in particular $\xi$ equals $\es$ on
  $[x_n+\delta_1 \varphi (n)-K_2, x_n+ \delta_1 \varphi (n)+K_2]$. For $x\in [x_n+\delta_1 \varphi (n)-1, x_n+ \delta_1 \varphi (n)+1]$,
  define
  \begin{equation}
    x' = \begin{cases}
      x_n + \delta_1\varphi (n) +K_1, \qquad
      &\text{if $x<x_n + \delta_1\varphi (n) +K_1$,}\\
      x_n + 2\delta_1\varphi (n) -K_1, \qquad
      &\text{if $x>x_n + 2\delta_1\varphi (n) -K_1$,}\\
      x,
      &\text{otherwise,}
    \end{cases}
  \end{equation}
  and set $I_i = [x'-K_i,x'-K_i]$, $i=1,2,$ so that $I_1 \subset I_2$.

  We now consider the BBMRE
  started at $x$ and for $k\ge 1$ we define
  \begin{equation}
    Z_k = | \{Y\in N(kT): Y_{lT} \in I_1\, \forall 1\le l\le K, Y_{s}\in
      I_2\, \forall s<kT\}|.
  \end{equation}
  $Z_k$ can be interpreted as the number of particles in the $k$-th generation  of a
  multi-type branching process; here, the type corresponds to the
  position of the particle in $I_1$ at which it is born (with exception of the initial
    particle which is at most at distance 1 from $I_1$), and where the
  number of offspring of a particle of type $y$ is distributed as
  $|\{Y\in N(T): Y_T\in I_1, Y_s\in I_2\, \forall s\le T\}|$ under
  $\mathtt P^\es_y$. In particular, using the Feynman-Kac formula as well as
  \eqref{eq:Kfix} and then \eqref{eq:Tfix}, the expected  offspring number
  of a particle of type $y$
  satisfies
  \begin{equation}
    \begin{split}
      \mathtt E^\es_y&[|\{Y\in N(T): Y_T\in I_1, Y_s\in I_2\, \forall s\le
          T\}|]
      \\&= e^{\es\, T} P_y(B_T\in I_1, B_s \in I_2 \,\forall s<T) \ge
      \frac 5 {16} e^{\es\, T} \ge e^{(1- \frac {\delta_2}2) \es \, T},
    \end{split}
  \end{equation}
  uniformly over all admissible types $y$. In addition,
  the second moment of the same quantity is finite, again uniformly
  over all admissible types, by comparison with branching process with
  branching rate $\es$. It thus follows by the standard
  results on multi-type branching processes  that for some
  $\rho \ge e^{(1- \frac {\delta_2}2) \es \, T}$ finite,
  $Z_k/\rho^k$ converges in distribution to a non-negative random variable
  $W$ with $P(W>0)>0$ (see e.g.~\cite[Theorem~14.1]{Harris63}, where $\rho $ is the principal
    eigenvalue of the expectation operator of the multi-type branching
    process; observe also that Condition 10.1 of this theorem is easily
    checked for $V$ being the Lebesgue measure). In particular, one can find $\varepsilon_2>0$ and
  $k_0$ large such
  that
  \begin{equation}
    \mathtt P_x^{\es}\big(Z_k \ge \varepsilon_2 e^{(1- \frac {\delta_2}2) \es
        \,k T} \big)
    \ge
    \mathtt P_x^{\es}(Z_k \ge \varepsilon_2 \rho^k )
    \ge \varepsilon_2
    \qquad \text{for all }k\ge k_0,
    \label{eq:phase2}
  \end{equation}
  uniformly in
  $x\in [x_n+\delta_1 \varphi (n)-1, x_n+ \delta_1 \varphi (n)+1]$. This terminates the investigation of the second phase of comparison with BRW, and we may now proceed to the proof of Claim~\ref{cl:many_part}.

  To this end, fix
  $K$ such that $(1-\varepsilon_2)^K < \varepsilon /4$ and set (for
    $\kappa $ and $t_1$ from \eqref{eq:ampli})
  \begin{equation}
    \label{eq:tprimefix}
      t' = \inf\{s\in [t_1,t], \kappa^s > K \vee (4/\varepsilon), t-s = k T \text{ for some }k\in
      \mathbb N \}.
  \end{equation}
  Observe that there is $c<\infty$ such that $t'<c $ for all $t\ge c$.
  Setting $\mathcal N = \{Y\in N(t'):Y_{t'}\in [r-1,r+1]\}$, we have,
  using \eqref{eq:ampli} and \eqref{eq:tprimefix} for the last inequality, that
  \begin{equation}
    \begin{split}
      \label{eq:claimaa}
      \ttP_r^{\xi}&\Big(\big |\big \{Y\in N_r(t): Y_{t}\in [x_n+\delta_1 \varphi (n),
            x_n + 2 \delta_1 \varphi (n) ] \big \} \big  |
        \le e^{(1-\delta_2)\es \, t}\Big)
      \\&\le
      \ttP^\xi_r(|\mathcal N|< \kappa^{t'}) +
      \ttP^\xi_r\big( \{|\mathcal N|\ge \kappa^{t'}\} \cap \mathcal A \big)
      \le
      \frac \varepsilon 4 +
      \ttP^\xi_r\big (\{|\mathcal N|\ge \kappa^{t'}\} \cap \mathcal A \big),
    \end{split}
  \end{equation}
  where $\mathcal A$ denotes the event that each particle in $\mathcal N$
  produces less than
  $e^{(1-\delta_2)\es\,t}$ particles in
  $[x_n+\delta_1 \varphi (n), x_n +2 \delta_1 \varphi (n)]$ at time $t$.
  For a particle at position $x\in [r-1,r+1]$, we then fix the intervals
  $I_1$, $I_2$ as above, and observe that the number of its children in
  $[x_n+\delta_1 \varphi (n), x_n +2 \delta_1 \varphi (n)]$ at time
  $t-t' =: k_t T$   dominates $Z_{k_t}$ under $\ttP^\es_x$. Since the
  offspring of different particles are independent, for $t$ large enough
  such that
  $e^{(1-\delta_2)\es\,t} \le \varepsilon_2 e^{(1-\frac {\delta_2}2)\es\,k_t T}$,
  we obtain
  \begin{equation}
    \begin{split}
      \ttP^\xi_r\big( \{|\mathcal N|\ge \kappa^{t'}\} \cap \mathcal A \big)
      &\le
      \ttE^\xi_r\Big[
      \prod_{Y\in \mathcal N} \ttP_{Y_{t'}} \big(Z_{k_t} \le
        e^{(1-\delta_2)\es\,t}\big);\, |\mathcal N|\ge \kappa^{t'}\Big]
      \\&\le
      \ttE^\xi_r\Big[
        \prod_{Y\in \mathcal N} \ttP_{Y_{t'}}\big( Z_{k_t} \le
          \varepsilon_2 e^{(1-\frac {\delta_2}2)\es\,k_t T}\big)
        ;\, |\mathcal N|\ge \kappa^{t'}\Big]
      \\&\le
      \ttE^\xi_r\Big[
        (1-\varepsilon_2)^{|\mathcal N|}
        ;|\mathcal N|\ge \kappa^{t'}\Big]
      \le (1-\varepsilon_2)^{\kappa^{t'}} \le (1-\varepsilon_2)^K \le
      \frac \varepsilon 4,
    \end{split}
  \end{equation}
  where for the third inequality we used \eqref{eq:phase2} and
  for the last two inequalities we applied \eqref{eq:tprimefix}.
  Combining \eqref{eq:claimaa} with the last display completes the proof
  of the claim.
\end{proof}

\subsection{Non-monotonicity of the solution to randomized F-KPP equation}

In this section we prove Theorem~\ref{thm:non-monotone}.
Its proof is based on the simple idea that if there are two
adjacent long stretches, the left one with potential $\ei$ and the right
one with $\es$, where the values of $w$ are comparable at some time $t_n$,
as proved in Theorem~\ref{thm:FKPPunBdd}, then at some later time $t_n+s$
the function $w$ must be non-monotone, since it grows faster on the
right stretch.
\newcommand{\Lipc}{H}
\newcommand{\expFK}{c}

\begin{proof} [Proof of Theorem~\ref{thm:non-monotone}]
  For every $\varepsilon>0$ we choose  $K=K(\varepsilon)$ such that
  \begin{equation}
    \label{eq:choice_K}
    f(K):= e^{\es} P_0\Big( \sup_{0\leq u\leq 1} |B_u|>K \Big) \leq \varepsilon.
  \end{equation}
  Recall that by Proposition~\ref{pro:coupling}, the definition of the
  coupling $\ttQ_{l,r}^\xi$ and the representation
  $w(t,x)=\ttP_{x}^\xi\big( N^\leq(t,0)\neq\emptyset \big)$ of the
  solution to \eqref{eq:KPP} (see Proposition~\ref{prop:KPPBRM}), for
  $\delta\in(0,1)$ there exist $l_n,r_n,t_n$ such that $t_n\to\infty$,
  $w(t_n,l_n)=\delta$, $r_n-l_n\tend{n}{\infty}\infty$ and such that for
  all $n$ large enough
  \begin{equation}
    \label{eq:mon_incr}
    \sup_{l\in[l_n-K,l_n+K]} w(t_n,l) \leq  \inf_{r\in[r_n-K,r_n+K]} w(t_n,r)+\varepsilon
  \end{equation}
  holds. We will prove the result by contradiction and therefore assume
  for the time being that the claim of the theorem does not hold. Then,
  for all $\varepsilon>0$, all $n$ large enough and all $s\in[0,1],$ we have
  \begin{equation}
    \label{eq:mon_decr1}
    \inf_{l\in[l_n-K,l_n+K]} w(t_n+s,l) \geq \sup_{r\in[r_n-K,r_n+K]} w(t_n+s,r)-\varepsilon.
  \end{equation}
  Let us choose $\varepsilon\in(0,\delta)$, $s'\in(0,1]$ small enough
  and $b\in(0,1)$ such that for all $s\in[0,s'],$
  \begin{align}
    \label{eq:choice_s'}
    e^{\es\, s}(\delta + 3\varepsilon) \leq b.
  \end{align}
  Recall that the solution can be represented by the Feynman-Kac formula \eqref{eq:feynman_kac_KPP} with some $F:[0,1]\to[0,1]$ fulfilling \eqref{eq:standard_cond} for some sequence $(p_k)$ fulfilling \eqref{eq:mom_offspr}.
  Let us abbreviate  $\expFK(w):=\frac{F(w)}{w}$, $w\in(0,1]$. It is easy to see that $\expFK$ is strictly decreasing, can be extended continuously to $w=0$, i.e.  $\expFK(0)=\lim_{w\downarrow0}c(w)=\sup_{w\in(0,1]}\expFK(w)=1$, $\expFK(1)=0$ and the function $\expFK:[0,1]\to[0,1]$ is Lipschitz continuous  with Lipschitz constant $\Lipc\in(0,\infty)$. Among others, due to \eqref{eq:mon_decr1} and  $w\in[0,1]$, for all $s\in[0,1]$   we have
  \begin{equation}
  	\label{eq:mon_decr}
  	\sup_{l\in[l_n-K,l_n+K]} \expFK\big(w(t_n+s,l)\big) \leq \inf_{r\in[r_n-K,r_n+K]} \expFK\big(w(t_n+s,r)\big)+\Lipc \varepsilon.
  \end{equation}
  Furthermore, by the Feynman-Kac formula \eqref{eq:feynman_kac_KPP} and the
  Markov property, for all $s\geq 0$ we have
  \begin{align*}
    w(t_n+s,l_n) &= E_{l_n} \Big[ \exp\Big\{ \int_0^{s}\xi(B_u)\expFK\big(w(t_n+s-u,B_u)\big)\diff u \Big\} w(t_n,B_{s}) \Big].
  \end{align*}
  Then due to $\xi\leq\es$, $w\in[0,1]$, $\expFK\leq1$, \eqref{eq:mon_incr},
  \eqref{eq:mon_decr1}, \eqref{eq:choice_K}, and \eqref{eq:choice_s'},
  for all   $n$ large enough we have for all  $s\in[0,s']$ that
  \begin{align}
    \label{eq:w_ngbh}
    w(t_n+s,l_n)&\leq  e^{\es \, s} \Big( P_{l_n}\big(\sup_{0\leq u\leq 1}|B_u-l_n|>K\big) + \sup_{l\in[l_n-K,l_n+K]} w(t_n,l) \Big) \leq b.
  \end{align}
  Furthermore, using $\xi\leq\es$, $w\in[0,1]$ and $\expFK(w)\in[0,1]$ for $w\in[0,1]$ we get that for all
  $s\in[0,1]$ we have
  \begin{align*}
    w(t_n+s,l_n) &\leq E_{l_n} \Big[ \exp\Big\{ \int_0^{s}\xi(B_u)\expFK\big(w(t_n+s-u,B_u)\big)\diff u \Big\} w(t_n,B_{s}); \sup_{0\leq u\leq 1} |B_u-l_n| \leq K \Big] \\
    &\qquad + e^\es P_0\Big( \sup_{0\leq u\leq 1} |B_u|>K \Big).
  \end{align*}
  To bound the first summand, we recall (by definition of $l_n,r_n$) that
  $\xi(l)=\ei$ for all $l\in[l_n-K,l_n+K]$ and $\xi(r)=\es$ for all
  $r\in[r_n-K,r_n+K]$. Using \eqref{eq:mon_incr} and \eqref{eq:mon_decr},
  we see that the first summand can be bounded from above by
  \begin{align*}
    &E_{l_n} \Big[ \exp\Big\{ \frac{\ei}{\es} \int_0^{s}\xi(B_u-l_n+r_n)\big(\expFK(w(t_n+s-u,B_u-l_n+r_n))+\Lipc\varepsilon\big)\diff u \Big\}\\
      &\qquad\qquad  \times \big(w(t_n,B_{s}-l_n+r_n)+\varepsilon \big); \sup_{0\leq u\leq 1} |B_u-l_n| \leq K \Big] \\
    &= e^{\ei \Lipc \, \varepsilon s} E_{r_n}  \Big[ \exp\Big\{ \frac{\ei}{\es} \int_0^{s}\xi(B_u)\expFK(w(t_n+s-u,B_u))\diff u \Big\} (w(t_n,B_{s})+\varepsilon); \sup_{0\leq u\leq 1} |B_u-r_n| \leq K \Big].
  \end{align*}
  Recall the inequality $e^{ax}\leq e^{x}-(1-a)x$ for all $a\in[0,1]$ and
  $x\geq0$. Then, since $\frac{\ei}{\es}\in(0,1)$, we get
  \begin{align} \label{eq:intermediate}
    \begin{split}
      w&(t_n+s,l_n) \leq f(K)   + e^{\ei \Lipc \, \varepsilon s}\left( \varepsilon e^{\ei s}+E_{r_n} \Big[ \exp\Big\{ \int_0^{s}\xi(B_u)\expFK(w(t_n+s-u,B_u))\diff u \Big\} w(t_n,B_{s}) \Big] \right. \\
        &\quad \left.- (1-\ei/\es) E_{r_n} \Big[  \int_0^{s}\xi(B_u)\expFK(w(t_n+s-u,B_u))\diff u \  w(t_n,B_{s}); \sup_{0\leq u\leq 1} |B_u-r_n|\leq K \Big]\right).
    \end{split}
  \end{align}
  Recalling \eqref{eq:mon_incr}, we also have
  $\inf_{r\in[r_n-K,r_n+K]} w(t_n,r) \geq \delta-\varepsilon$. Furthermore, using the properties of $\expFK$, for $\varepsilon$ small enough such that $\varepsilon+b<1$, we have that $\underline{\expFK}=\underline{\expFK}(\varepsilon,b):=\inf_{v\in[0,b+\varepsilon]}\expFK(v)>0$.  Using
  \eqref{eq:choice_K}, $\xi\geq\ei$, \eqref{eq:mon_decr1},
  \eqref{eq:w_ngbh}, the inequality $e^{x}\leq 1+2x$ for $x\geq 0$ small
  enough, and $w\in[0,1]$, we get, choosing $s=s'$ from
  \eqref{eq:choice_s'} and continuing the bound from
  \eqref{eq:intermediate},
  \begin{align*}
    w(t_n+s',l_n)&\leq \varepsilon(1+e^{(1+\Lipc\varepsilon)\ei s'}) + (1+2\Lipc\varepsilon\ei s') w(t_n+s',r_n) -  (1-\ei/\es)\, \ei\, \underline{\expFK}\,(\delta-\varepsilon)(1-\varepsilon) s' \\
    &\leq w(t_n+s',r_n) +  \varepsilon(1+2\ei (1+2\Lipc\varepsilon)) - (1-\ei/\es)\, \ei\, \underline{\expFK}\,(\delta-\varepsilon)(1-\varepsilon) s'
  \end{align*}
  and the right-hand side can made  smaller than
  $w(t_n+s',r_n)-2\varepsilon$ if we choose $s'$ (say) of order
  $\sqrt{\varepsilon}$ and  $\varepsilon$ small enough. But this is a
  contradiction to \eqref{eq:mon_decr1}, which hence proves
  Theorem~\ref{thm:non-monotone}.
\end{proof}

\appendix

\section{Appendix: Further auxiliary results}
We collect here a couple of results needed primarily for the proof of
Lemma \ref{le:space_perturb}, and start with several lemmas concerning  the logarithmic moment generating
functions defined in \eqref{eq:def_L_quant} as well as related objects.
They are proved in \cite{DreSchmi19} and are
modifications of the corresponding discrete-space
statements proved in \cite{CeDr-15}.

\begin{lemma}[{\cite[Lemma~\ref{LD-le:log mom fct and derivatives}]{DreSchmi19}}]
  \label{le:quenched_lmgf}
  We recall that $P_x^{\zeta,\eta}$ has been defined  in \eqref{eq:def_Px}.
  \begin{enumerate}
    \item The functions $L$, $L_x^\zeta$, and $\overline{L}_x^\zeta$, for
    $x\in\R$, defined in \eqref{eq:def_L_quant}, are infinitely differentiable on
    $(-\infty,0)$. Furthermore, for all $\eta<0$ we have
    \begin{align}
      \big( L_x^\zeta \big)'(\eta)             & = \frac{E_x\Big[
          e^{\int_0^{H_{\lceil x\rceil-1}}(\zeta(B_r)+\eta)\diff r}H_{\lceil x\rceil-1}
          \Big]}{E_x\Big[ e^{\int_0^{H_{\lceil x\rceil-1}}(\zeta(B_r)+\eta)\diff r}
          \Big]} = E_x^{\zeta,\eta}[ \tau_{\lceil x\rceil-1} ],\quad
      x\in\R,\label{eq:L_x'_tight}
      \\
      \big( \overline{L}_x^\zeta \big)'(\eta)  & =
      \frac{1}{x}E_x^{\zeta,\eta}\big[ H_0 \big],\quad x>0,\label{eq:Lbar_x'_tight}
      \\
      L'(\eta)                                 & =\E\Big[ \frac{E_1\big[
            e^{\int_0^{H_0}(\zeta(B_r)+\eta)\diff r}H_0 \big]}{E_1\big[
            e^{\int_0^{H_0}(\zeta(B_r)+\eta)\diff r} \big]} \Big] = \E\big[
        E_1^{\zeta,\eta}[ H_0 ] \big],\label{eq:L'_tight}
    \end{align}
    and 
    \begin{align}
      \big( L_x^\zeta \big)''(\eta)            
      & =E_x^{\zeta,\eta}\big[
        \tau_{\lceil x\rceil-1}^2 \big] - \big( E_x^{\zeta,\eta}[ \tau_{\lceil
            x\rceil-1} ]\big)^2 =\text{\emph{Var}}_x^{\zeta,\eta}(\tau_{\lceil x\rceil-1})
      >0 ,\quad x\in\R,\label{eq:L_x''_tight}
      \\
      \big( \overline{L}_x^\zeta \big)''(\eta) &
      =\frac{1}{x}\text{\emph{Var}}_x^{\zeta,\eta}(H_0),\quad
      x>0,\label{eq:Lbar_x''_tight}
      \\
      L''(\eta)                                
      & =\E\Big[
        E_1^{\zeta,\eta} [ H_0^2 ] - \big( E_1^{\zeta,\eta} [ H_0] \big)^2 \Big] =
      \E\big[ \text{\emph{Var}}_1^{\zeta,\eta}( H_0 ) \big] >0.\label{eq:L''_tight}
    \end{align}

    \item For each compact interval $\triangle\subset(-\infty,0)$, there exists
    a constant $\Cl{constDerLMG}=\Cr{constDerLMG}(\triangle)>1$, such that the
    following inequalities hold $\p$-a.s.:
    \begin{align*}
      -\Cr{constDerLMG} \leq \inf_{\eta\in\triangle,x\geq 1} \big\{
        L_{\lfloor x\rfloor}^\zeta(\eta),\overline{L}_x^\zeta(\eta),L(\eta)  \big\}
      & \leq \sup_{\eta\in\triangle,x\geq 1}\big\{ L_{\lfloor
          x\rfloor}^\zeta(\eta),\overline{L}_x^\zeta(\eta),L(\eta) \big\} \leq
      -\Cr{constDerLMG}^{-1},                                                     \\
      \Cr{constDerLMG}^{-1}\leq \inf_{\eta\in\triangle,x\geq 1} \big\{
        (L_{\lfloor x\rfloor}^\zeta)'(\eta),(\overline{L}_x^\zeta)'(\eta),L'(\eta)
        \big\} & \leq \sup_{\eta\in\triangle,x\geq 1}  \big\{ (L_{\lfloor
            x\rfloor}^\zeta)'(\eta),(\overline{L}_x^\zeta)'(\eta),L'(\eta) \big\}  \leq
      \Cr{constDerLMG},                                                 \\
      \Cr{constDerLMG}^{-1}\leq \inf_{\eta\in\triangle,x\geq 1} \big\{
        (L_{\lfloor x\rfloor}^\zeta)''(\eta),(\overline{L}_x^\zeta)''(\eta),L''(\eta)
        \big\} & \leq \sup_{\eta\in\triangle,x\geq 1}  \big\{ (L_{\lfloor
            x\rfloor}^\zeta)''(\eta),(\overline{L}_x^\zeta)''(\eta),L''(\eta) \big\}  \leq
      \Cr{constDerLMG}.
    \end{align*}
  \end{enumerate}
\end{lemma}

\begin{lemma}[{\cite[Lemma~\ref{LD-le:expectedLogMGfct}]{DreSchmi19}}] \label{lem:vcetc}
  \begin{enumerate}
    \item
    The function $(-\infty,0)\ni\eta\mapsto L(\eta)$ is infinitely
    differentiable and its derivative $L'(\eta)$ is positive and
    monotonically strictly increasing.

    \item
    We have $\p\text{-a.s.}$ that
    \begin{equation}
      \label{eq:LLNEmpLogMG_tight}
      \lim_{x\to\infty} \overline{L}_{x}^\zeta(\eta) =
      L(\eta)\quad\text{for all }\eta< 0.
    \end{equation}

    \item $L'(\eta)\downarrow 0$ as $\eta\downarrow -\infty$

    \item For every $v>v_c:=\frac{1}{L'(0-)}$
    (where $\frac{1}{+\infty}:=0$), which we call \emph{critical
      velocity}, there exists a
    \begin{equation}
      \label{eq:legTraExpectedLogMGfct_1_tight}
      \text{ unique solution $\overline{\eta}(v)<0$ to the equation }
      L'( \overline{\eta}(v) ) = \frac{1}{v}.
    \end{equation}
    $\overline{\eta}(v)$ can be characterized as the unique maximizer to
    $(-\infty,0]\ni \eta\mapsto \frac{\eta}{v} - L(\eta) $, i.e.
    \begin{equation}\label{eq:legTraExpectedLogMGfct_2_tight}
      \sup_{\eta\leq 0} \Big( \frac{\eta}{v} - L(\eta) \Big)=
      \frac{\overline{\eta}(v)}{v} - L\big( \overline{\eta}(v) \big).
    \end{equation}
    The function $(v_c,\infty)\ni v\mapsto \overline{\eta}(v)$ is
    continuously differentiable and strictly decreasing.
  \end{enumerate}
  \label{le:annealed_lgmf}
\end{lemma}

We now recall the well-known existence
of the Lyapunov exponent for the solutions to \eqref{eq:PAM}.

\begin{proposition}[{\cite[Proposition~\ref{LD-prop:lyapunov}, Corollary~\ref{LD-cor:lyapunov_alt}]{DreSchmi19}}]
	\label{prop:lyapunov_tight} Assume \eqref{eq:POT}--\eqref{eq:PAM_initial}.
	For all $v\geq0$ and all $u_0\in\mathcal{I}_{\text{PAM}}$ the limit
	\begin{equation}
		\label{eq:lyapunov_alt}
		\Lambda (v) = \lim_{t\to\infty} \frac{1}{t} \ln u^{u_0}(t,vt)
	\end{equation}
        exists $\p$-a.s., is non-random and independent of $u_0$.
        We have $\Lambda(0)=\es$, $\Lambda $ is nondecreasing, linear
        on $[0,v_c]$, strictly concave on $(v_c,\infty)$ and
        $\lim_{v\to\infty} \frac{\Lambda(v)}{v}=-\infty$. In particular,
        there exists a unique $v_0>0$ such that $\Lambda(v_0)=0$. Furthermore, the convergence in \eqref{eq:lyapunov_alt} holds uniformly on any compact interval $K\subset[0,\infty)$. 
\end{proposition}

\begin{lemma}[{\cite[Lemma~\ref{LD-le:concEtaEmpLT} (b)]{DreSchmi19}}]
  \begin{enumerate}
    \item
    For every $v>v_c$ there exists a finite random variable
    $\mathcal{N}=\mathcal{N}(v)$ such that for all $x\geq\mathcal{N}$
    the solution $\eta_x^\zeta(v)<0$ to
    $E_x^{\zeta,\eta_x^\zeta(v)}[H_0]=\frac{x}{v}$  exists.

    \item For each $q\in\mathbb{N}$ and each compact interval $V\subset(v_c,\infty)$,
    there exists
    $\Cl{const:concEtaEmpLT}:=\Cr{const:concEtaEmpLT}(V,q)\in(0,\infty)$ such that
    \begin{equation}
      \label{eq:concEtaEmpLT_cont_tight}
      \p\Big( \sup_{v\in V}\sup_{x\in[n,n+1)} |  \eta_x^\zeta(v) -
        \overline{\eta}(v)| \geq \Cr{const:concEtaEmpLT}\sqrt{\frac{\ln n}{n}} \Big)
      \leq \Cr{const:concEtaEmpLT} n^{-q}\qquad \text{for all }n\in\mathbb{N}.
    \end{equation}
  \end{enumerate}
  \label{le:concEtaEmpLT_tight}
\end{lemma}

\begin{lemma}[{\cite[Lemma~\ref{LD-le:PertEta_n}]{DreSchmi19}}]
  \label{le:PertEta_n_tight}
  There exists a constant $\Cl{const_eta_pert}>0$ such that  $\p$-a.s., for all
  $x \in (0,\infty)$ large enough, uniformly in $v\in V$ and $0\leq h\leq x$,
  \begin{equation}
    \label{eq:PertEta_n_tight}
    \big| \eta_{x}^\zeta(v) - \eta_{x+h}^\zeta(v)   | \leq
    \Cr{const_eta_pert}\frac{h}{x}.
  \end{equation}
\end{lemma}

In the final lemma we recall a Hoeffding-type inequality for mixing
random variables, which is a consequence of {\cite[Theorem 2.4]{rio}}.

\begin{lemma}[{\cite[Corollary~\ref{LD-cor:rioDepExp}]{DreSchmi19}}]
  \label{lem:rioDepExp_tight}
  Let $(Y_i)_{i\in\mathbb{Z}}$ be a sequence of real-valued bounded random
  variables, $\widetilde{\F}^k:=\sigma(Y_j:j\geq k)$, and let
  $(m_1,\ldots,m_{n})$ be an $n$-tuple of positive real numbers such that
  for all $i\in\{1,\ldots,n\}$,
  \begin{equation*}
    \sup_{j \in \{1, \ldots, i\}}\Big( \|Y_i^2\|_\infty + 2 \Big \Vert
      Y_i\sum_{k=j}^{i-1}\E[Y_k|{\widetilde{\F}}^i]\Big \Vert_\infty \Big) \leq m_i,
  \end{equation*}
  with the convention $\sum_{k=i}^{i-1}\E[Y_k|\widetilde{\F}^i]=0$. Then for
  every $x>0$,
  \begin{equation*}
    \p\Big( \Big| \sum_{i=1}^{n} Y_i \Big|\geq x \Big) \leq \sqrt{e}\exp\left\{
      -x^2/(2m_1+\cdots+2m_{n}) \right\}.
  \end{equation*}
\end{lemma}

\section{Appendix: Non-triviality of the regime of validity} \label{sec:nonTrivial}

The next lemma is used to show that there are potentials $\xi $ that simultaneously
satisfy the assumptions of Theorem~\ref{thm:PAMBdd} as well as of
Theorems~\ref{thm:FKPPunBdd} and~\ref{thm:non-monotone}.

\begin{lemma}
  \label{lem:Cxi}
  Let $\xi$ be the potential constructed in \eqref{eq:def_xi} for real
  numbers $\es$ and $\ei$ satisfying $0 < \ei < \es$ (with
    \eqref{eq:relation_es_ei} not necessarily fulfilled). Then, making
  the dependence of $L$ explicit in writing $L= L_{\xi},$ we have that
  the family of real numbers $\frac{1}{L_{C \xi}'(0-)},$
  $C \in [1,\infty),$ is upper bounded away from infinity.
\end{lemma}
\begin{proof}
  Equation \eqref{eq:L'_tight}
  and monotone convergence entail that for all $C \in [1,\infty)$ we have
  \begin{align*}
    L_{C\xi}'(0-)  &=\E\Bigg[ \frac{E_1\big[
          e^{C\int_0^{H_0}(\xi(B_r) - \es)\diff r}H_0 \big]}{E_1\big[
          e^{C\int_0^{H_0}(\xi(B_r) - \es)\diff r} \big]} \Bigg].
  \end{align*}
  Since the expectation in the denominator on the right-hand side of the
  previous display is $\P$-a.s.\ upper bounded by $1$, we can continue
  the above to infer that for some positive constant $c>0$ and all
  $C \in [1,\infty)$ we have
  \begin{align*}
    L_{C\xi}'(0-)  &\ge             \E\Big[ E_1\big[
        e^{C\int_0^{H_0}(\xi(B_r) - \es)\diff r}H_0 \big] \cdot
      \mathds{1}_{\{\xi (x) = \es \, \forall x \in [0,2]\}}\Big]\\
    &\ge             \E\Big[ E_1[ H_0 \cdot \mathds{1}_{\{B_r \in [0,2] \, \forall r \in [0,H_0]\}}] \cdot \mathds{1}_{\{\xi (x) = \es \, \forall x \in [0,2]\}}\Big] \ge c > 0,
  \end{align*}
  which finishes the proof of the lemma.
\end{proof}

\begin{proposition} \label{prop:nonTriv}
There exist potentials  $\xi $ that
satisfy the assumptions of Theorem~\ref{thm:PAMBdd} as well as of
Theorems~\ref{thm:FKPPunBdd} and~\ref{thm:non-monotone}.
\end{proposition} 
\begin{proof}
It is sufficient to find a potential $\xi$ as in 
\eqref{eq:def_xi}, under the sole assumption $\es/\ei >2$ of
\eqref{eq:relation_es_ei}, such that at the same time \eqref{eq:VEL} holds true for the respective potential. 

For this purpose, we choose an arbitrary potential $\xi$ as in 
\eqref{eq:def_xi} satisfying
\eqref{eq:relation_es_ei}
We then infer that for such a potential and $C \in (0,\infty)$ large enough,
one has---making explicit the dependence of the respective quantities on
the potential---that $v_0(C \xi) > v_c(C \xi).$
Indeed, note that Lemma \ref{lem:vcetc} entails
 $v_c(\xi)=\frac{1}{L'(0-)},$ 
and Lemma
\ref{lem:Cxi} implies that $\frac{1}{L'(0-)}$ is upper bounded away from
infinity for the potentials $C \xi$ as $C \to \infty.$ Regarding $v_0,$ a comparison
with the constant potentials $C \ei$ yields that $v_0(C\xi) \to \infty$
as $C \to \infty,$ so \eqref{eq:VEL}  holds true for all $C$ large enough, which is sufficient for the assumptions of Theorem~\ref{thm:PAMBdd} to be fulfilled.

 At the same time, the potential $C \xi$ still
satisfies \eqref{eq:relation_es_ei} and hence fulfills the assumptions of Theorems~\ref{thm:FKPPunBdd} and~\ref{thm:non-monotone}. 
\end{proof}

\bibliography{bibliothek}
\bibliographystyle{alpha}

\end{document}